\documentclass[12pt,leqno]{amsart}
\usepackage{amsmath, amssymb, amscd, amsfonts,    stmaryrd, turnstile, mathrsfs, eucal}

\newtheorem{theorem}{Theorem}[section]
\newtheorem{lemma}[theorem]{Lemma}
\newtheorem{proposition}[theorem]{Proposition}
\newtheorem{corollary}[theorem]{Corollary}

\theoremstyle{definition}
\newtheorem{definition}[theorem]{Definition}
\newtheorem{example}[theorem]{Example}
\newtheorem{remark}[theorem]{Remark}

\def\car#1,#2,#3,#4{
$
   \CD
   #1           @>{}>>          #2        \\
   @V{{}}VV                  @VV{{}}V  \\
   #3        @>{{}}>>   #4
   \endCD
   $}

\def\Car#1,#2,#3,#4,#5,#6,#7,#8{
$$
   \CD
   #1           @>#2>>          #3        \\
   @V{#4}VV                  @VV{#5}V  \\
   #6        @>{#7}>>   #8
   \endCD
   $$}

\begin{document}

\title[AROUND  PR\"UFER EXRENSIONS ]{AROUND  PR\"UFER EXTENSIONS OF RINGS}

\author[G. Picavet and M. Picavet]{Gabriel Picavet and Martine Picavet-L'Hermitte}
\address{Math\'ematiques \\
8 Rue du Forez, 63670 - Le Cendre\\
 France}
\email{picavet.mathu (at) orange.fr}

\begin{abstract} The paper intends to apply the properties of Pr\"ufer extensions, investigated in the Knebusch-Zhang book,  to ring extensions $R\subseteq  S$.
The integral closure $\overline R$ of $R$ in $S$ is shown to be the intersection of all $T\in [R,S]$, such that $T\subseteq S$ is Pr\"ufer. We are then able to establish an avoidance lemma for integrally closed subextensions. Rings of sections of the affine scheme defined by $R$ provide results on $S$-regular ideals. Some results on pullbacks characterizations of Pr\"ufer extensions are given. We introduce locally strong divisors, examining the properties of strong divisors of a local ring and their links with Pr\"ufer extensions. The locally strong divisors allow us to give characterizations of QR-extensions. We apply our results to Nagata extensions of rings. We also look at the Pr\"ufer hull of a  Nagata extension. We define quasi-Pr\"uferian rings that may differ from quasi-Pr\"ufer integral domains. We then derive some results on minimal and FCP extensions. Finally, we study the set of all primitive elements in an extension.
\end{abstract} 



\maketitle

\section{Introduction and Notation}

We consider the category of commutative and unital rings, whose   flat epimorphisms will be  strongly  involved, like localizations with respect to a multiplicatively closed  subset.  

If $R\subseteq S$ is a (ring) extension, we denote by $[R,S]$ the set of all $R$-subalgebras of $S$ and by $[R,S]_{fg}$ the set of all $T\in[R,S]$, such that $T$ is of finite type over $R$.
 Any undefined material is explained in the next subsection and in the following  sections.

\subsection{An overview of the paper}
We present some properties of Pr\"ufer extensions of rings and derive from them new results, using the properties and definitions of Knebusch and Zhang \cite{KZ}. It is well known that Pr\"ufer extensions are nothing but normal pairs. Pr\"ufer extensions are defined by flat epimorphims, while normal pairs are defined by the integrally closed property. We will deal with the Pr\"ufer aspect, except in  Section 6. 

In Section 2 we give some recalls about Pr\"ufer extensions. We also give rules on flat epimorphisms and direct limits, in order to make easier our proofs.

If $R\subset S$ is a ring extension, an ideal $I$ of $R$ is called $S$-regular by \cite{KZ} if $IS=S$. Such ideals are a useful concept in the next sections. Moreover, these ideals $I$ permit to factorize the extension through the ring of sections defined by the open subset associated to $I$. This is developed in Section 3, with some applications to Pr\"ufer extensions. By the way, we give rules  allowing  to calculate rings of sections.

In Section 4, we show that the integral closure $\overline R$ of a ring extension $R\subset S$ is the set intersection of all $T\in [R,S]$ such that $T\subseteq S$ is Pr\"ufer. This statement generalizes a classical result on integral closures.

As an application, we show that an avoidance lemma with respect to finitely many integrally closed subrings holds. The proof is not easy and uses Manis valuations. We also show an avoidance lemma with respect to finitely many flat epimorphisms. This is explained in Section 5.

Section 6 deals with pullbacks results. Olivier proved that integrally closed extensions are characterized by pullbacks in which some morphism is of the form $V \to K$, where $V$ is a semi-hereditary ring with total quotient ring $K$.  We adapt this result to the Pr\"ufer case and evidently reuse the normal pairs. Another result concerns a more classical situation.

In Section 7, we deal with extensions over local rings and introduce the strong divisors considered by \cite{KZ}. A strong divisor is a regular element $t$ of a ring  $R$, such that $Rt$ is comparable to each ideal of $R$. The maximal Pr\"ufer extension of a local ring $R$ is the localization of $R$ with respect to the multiplicatively closed subset of all strong divisors of $R$. We develop a theory of strong divisors. The most striking results are that a regular element $x$ of a local ring is a strong divisor if and only if $R\to R_x$ is Pr\"ufer, and that  an extension of finite type of $R$ is Pr\"ufer if and only if it is of the form $R\to R_x$, where $x$ is a strong divisor.

QR-extensions $R\subset S$ are studied in Section 8: they are extensions such that each $T\in [R,S]$ is (isomorphic to) a localization. They are evidently Pr\"ufer. We also look at the Bezout extensions of \cite{KZ} and examine the Bezout and Pr\"ufer hull of an extension. Over a local ring or a Nagata ring $R(X)$, the Pr\"ufer, Bezout and and QR properties are equivalent. To go further we have introduced locally strong divisors. As locally strong divisors appear each time we are dealing with Pr\"ufer extensions, we see that a ring $R$ admits non trivial Pr\"ufer extensions if $R$ has locally strong divisors that are non-units. This concept is more stable than that of strong divisors. An interesting result is that QR-extensions are characterized by using locally strong divisors. Another one is that a QR-extension $R\subset S$ verifies that for each $s\in S$ there is a locally strong divisor $\rho$, such that $\rho s \in R$. The section ends on extensions whose supports are finite.

Section 9 is concerned with Nagata extensions $R(X) \subset S(X)$. We show that such an extension is Pr\"ufer if and only if $R\subset S$ is Pr\"ufer. We already said that over a Nagata ring the Pr\"ufer and QR-concepts coincide. 
The Pr\"ufer hull $\widetilde R$ of an extension gives $\widetilde R(X)$ for its Nagata extension in a lot of cases. It may be that the result holds for any extension but we do not know the answer.   When $R$ is a local ring, we show that the strong divisors of $R(X)$ are in some sense the strong divisors  of $R$.

We define  in Section 10 quasi-Pr\"uferian rings as rings $R$ such that  $R\to R(X)$ is an $i$-extension. But our definition is not equivalent to $R\to \mathrm{Tot}(R)$ is quasi-Pr\"ufer in the sense of \cite{Pic 5}, contrary to the integral domains context, where we recover the classical notion of quasi-Pr\"ufer rings. We get some results about these rings, largely inspired by the integral domain context.
A sequence of statements in a theorem shows that a ring is quasi-Pr\"ufer if it is quasi-Pr\"uferian. The converse holds if the total quotient ring of the ring is zero-dimensional.

Section 11 is devoted to minimal or FCP extensions of a local ring that are either Pr\"ufer or have the QR-property. A special attention is paid to $\mathcal  B$-extensions (extensions that are locally determined in some sense).

The paper ends by  considering the set of all primitive  elements in an extension, a study initiated  by Dobbs and Houston. There is a link with quasi-Pr\"ufer extensions. 

\subsection{ Basics concepts}
    
As usual, Spec$(R)$ and Max$(R)$ are the set of prime and maximal ideals of a ring $R$ and $\mathrm U(R)$ is the set of all its units.
 
We now give some notation for a ring morphism $f: R\to S$. We denote by ${}^af$ the spectral map $\mathrm{Spec}(S) \to \mathrm{Spec}(R)$.
Then $\mathcal X_R (S)$ (or $\mathcal X (S)$) is the image of the map ${}^af$  and we say that $f$ is an {\it i-morphism} if ${}^af$ is injective. If $Q$ is a prime ideal of $S$ lying over $P$ in $R$, the ring morphism $R_P\to S_Q$ is called the local morphism  at $Q$ of the morphism. 
 
Then $(R:S)$ is the conductor of an extension $R\subseteq S$. The integral closure of $R$ in $S$ is denoted by $\overline R^S$ (or by $\overline R$ if no confusion can occur).
 
A {\it local} ring is here what is called elsewhere a quasi-local ring. For an extension $R\subseteq S$ and an ideal $I$ of $R$, we write $\mathrm{V}_S(I):=\{P\in\mathrm{Spec}(S)\mid I\subseteq P\}$ and $\mathrm D_ S(I)$ for its complement. If $R$ is a ring; then $\mathrm Z (R)$ denotes the set of all its zero-divisors.
  The support of an $R$-module $E$ is $\mathrm{Supp}_R(E):=\{P\in\mathrm{Spec}(R)\mid E_P\neq 0\}$, and $\mathrm{MSupp}_R(E):=\mathrm{Supp}_R(E)\cap\mathrm{Max}(R)$. When $R\subseteq S$ is an extension, we will set $\mathrm{Supp}(T/R):=\mathrm{Supp}_R(T/R)$ and $\mathrm{Supp}(S/T):=\mathrm{Supp}_R(S/T)$  for each $T\in [R,S]$, unless otherwise specified.
   
If $R\subseteq S$ is a ring extension and $\Sigma$ a {\em mcs} of $R$ ({\it i.e.} a multiplicatively closed subset of $R$), then $S_\Sigma$ is both the localization $S_{\Sigma}$ as a ring and the localization at $\Sigma$ of the $R$-module $S$; that is, $S\otimes_R R_\Sigma$.  
 
Let $\Sigma_1$ and $\Sigma_2$ be two mcs of a ring $R$. We denote by $\Sigma_2/1$ the image of $\Sigma_2$ in $R_{\Sigma_1}$. We recall that $R_{\Sigma_1 \Sigma_2}=(R_{\Sigma_1})_{\Sigma_2/1}$. It follows that if $x\in R$ and $\Sigma$ is a mcs of $R$, then $(R_x)_\Sigma = (R_\Sigma)_{x/1}$. 
 
Flat epimorphisms  and their properties are the main tool in this paper. We use the theory that was developed by D. Lazard \cite[Chapter IV]{L}. The reader may also use the scholium of our paper \cite{Pic 5}.

When $R\to S$ and $R\to T$ are ring morphisms, we will write $S\cong_R T $, (or $S\cong T$) if there is an isomorphism of $R$-algebras $S \to T$. It may happens that $\cong$ is  replaced with $=$.

Let  $R\subseteq S$ be an extension. A {\it chain} of $R$-subalgebras of $S$ is a set of elements of $[R,S]$ that are pairwise comparable with respect to inclusion.   
 We say that $R\subseteq S$ is {\it chained} if $[R,S]$ is a chain. We also say that the extension has FCP (or is an FCP extension) if each chain in $[R,S]$ is finite, or equivalently, the poset $[R,S]$ is Artinian and Noetherian. 
An extension is called FIP if $[R,S]$ has finitely many elements. An extension $R\subset S$ is called minimal if $[R,S] =\{R,S\}$.
According to \cite[Th\'eor\`eme 2.2]{FO}, a minimal extension is either integral or a flat epimorphism. 
Finally, $|X|$ is the cardinality of a set $X$, $\subset$ denotes proper inclusion  (contrary to \cite{KZ} where $\subset$ denotes the large inclusion). A compact topological space does not need to be separated. For a positive integer $n$, we set $\mathbb{N}_n:=\{1,\ldots,n\}$. 

 \section{Some definitions, notation and useful results} 
An extension $R\subseteq S$ is called {\it Pr\"ufer} if $R\subseteq T$ is a flat epimorphism for each $T\in[R,S]$ (or equivalently, if $R\subseteq S$ is a normal pair) \cite[Theorem 5.2, p. 47]{KZ}. A Pr\"ufer integral extension is trivial.
  
We denote by $Q(R)$ the complete ring of quotients (Utumi-Lambeck) of a ring $R$.
  
  \begin{definition}\label{A0.7} \cite{KZ} A ring extension $R\subseteq S$ has:
\begin{enumerate}
\item  a greatest flat epimorphic subextension  $R \subseteq \widehat R^S$, called the {\bf Morita hull} of $R$ in  $S$. 
\item  a  greatest Pr\"ufer subextension $R \subseteq \widetilde R^S$, called the {\bf Pr\"ufer hull}  of $R$ in $S$. 

\end{enumerate}
 We set  $\widehat{R} := \widehat R^S$ and $\widetilde{R} := \widetilde R^S$, if no confusion can occur.
 \end{definition}
 
 A ring $R$ has: 
 
(1) \cite{L} a maximal flat epimorphic extension $R\subseteq\mathbb M(R):=  {\widehat R}^{\mathrm{Q}(R)}$ (also termed the maximal flat epimorphic extension by some authors, like \cite{L}).
 
(2) \cite{KZ} a maximal Pr\"ufer extension $R\subseteq\mathbb P(R):=  {\widetilde R}^{\mathrm{Q}(R)}$.
 
 Note that $\widetilde R^S$ is denoted by $\mathrm P(R,S)$ in \cite{KZ} and $\widehat R^S$ coincide with the weakly surjective hull $\mathrm{M}(R,S)$ of \cite{KZ}. Our terminology is justified because Morita's construction is earlier \cite[Corollary 3.4]{M}.  The Morita hull can be computed by using a (transfinite) induction \cite{M}. Let $S'$ be the set of   all $s\in S$,  such that there is some ideal $I$ of $R$, such that $IS= S$ and $Is \subseteq R$. 
Then $R\subseteq S'$ is a subextension of $R\subseteq S$. We set $S_1:= S'$ and  $S_{i+1}:= (S_i)' \subseteq S_i$. By \cite[p.36]{M}, if $R \subset S$ is an FCP extension, then $\widehat R = S_n$ for some integer $n$. 

We also note the  following known consequence:

\begin{proposition}\label{CAREPIPLAT} An extension $R\subseteq S$ is a flat epimorphism if and only if for each $s\in S$ there is some ideal $I$ of $R$ such that $IS=S$  and $Is \subseteq  R$ (or equivalently ($R:_Rs)S =S$).
\end{proposition} 

\begin{corollary} An extension $R\subseteq S$ is Pr\"ufer if and only if $R[s] =(R:_Rs)R[s]$     for each $s\in S$.
\end{corollary}
\begin{proof} Use the definition of Pr\"ufer extensions by flat epimorphisms.
\end{proof}

If an extension $R\subseteq S$ is Pr\"ufer and $\Sigma$ is a mcs of $R$, then $R_\Sigma \subseteq S_\Sigma$ is Pr\"ufer.  We have a converse.

\begin{proposition}\label{LOC} \cite[Proposition 1.1]{Pic 5} An extension $R\subset S$ is Pr\"ufer if and only if $R_M\subseteq S_M$ is Pr\"ufer for each $M \in \mathrm{Max}(R)$ (resp.; for each $M\in \mathrm{Spec}(R)$).
\end{proposition}

\begin{proposition} \label{FAITHPRUF} \cite[Corollary 3.15]{Pic 5} Pr\"ufer extensions are descended by faithfully flat morphisms.
\end{proposition}

\begin{proposition} The Pr\"ufer property of extensions $R\subseteq S$ is local on the spectrum; that is if $\mathrm{Spec}(R)=\mathrm D(r_1)\cup\cdots\cup
\mathrm D(r_n)$ for some elements $r_1,\ldots,r_n\in R$ and $R_{r_i} \subseteq S_{r_i}$ is Pr\"ufer for each $i=1,\ldots,n$, then $R\subseteq S$ is Pr\"ufer.
\end{proposition}
\begin{proof} The extension $R_{r_1}\times\cdots\times R_{r_n}\subseteq S_{r_1}\times\cdots\times S_{r_n}$ is Pr\"ufer \cite[Proposition 5.20, p.56]{KZ}. To conclude use Proposition~\ref{FAITHPRUF} since $R\to R_{r_1}\times \cdots \times R_{r_n} $ is faithfully flat.
\end{proof}

In \cite{Pic 5}, a minimal flat epimorphism is called a {\it Pr\"ufer minimal} extension. An FCP Pr\"ufer extension has FIP and is a tower of finitely many Pr\"ufer minimal extensions \cite[Proposition 1.3]{Pic 5}.

In \cite{Pic 5}, we defined an extension $R\subseteq S$ to be {\it quasi-Pr\"ufer} if it can be factored $R\subseteq R'\subseteq S$, where $R\subseteq R'$ is integral and $R'\subseteq S$ is Pr\"ufer. In this case $R'$ is necessarily $\overline R$. An FCP extension is quasi-Pr\"ufer \cite[Corollary 3.4]{Pic 5}. 

An extension $R\subseteq S$ is called {\it almost-Pr\"ufer} if it can be factored $R\subseteq S'\subseteq S$, where the first extension is Pr\"ufer and the second is integral. In this case $S'$ is necessarily $\widetilde R$. An almost-Pr\"ufer extension is quasi-Pr\"ufer \cite{Pic 5}.

We now give some rules on flat epimorphisms. The following result of Lazard is a key result. Let $R$ be a ring. We denote by $\mathcal {FE}$ the collection of classes up to an isomorphism of flat epimorphisms whose domain is $R$ and by $\mathcal X$ the set of subsets of $\mathrm{Spec}(R)$ that  are affine schemes,  when endowed with the  induced sheave. The elements of $\mathcal X$ are  compact and stable under generization. 

\begin{proposition}\label{LAZZEPI} \cite[Proposition 2.5, p.112]{L} The map $\mathcal{FE}\to\mathrm{Spec}(R)$, defined by $T\mapsto \mathcal X (T)$  is a bijection onto $\mathcal X$. The inverse map is as follows: an affine scheme $X$ of $\mathrm{Spec}(R)$ gives $R\to\Gamma(X)$, the ring of sections over~$X$.
\end{proposition}

The next result, proved in \cite[Proposition 3.4.10, p.242]{EGA}, will be useful in the sequel.

\begin{proposition}\label{LIM} $(\mathcal L$)-rule Let $R\to E$ be a ring morphism and $E = \varinjlim E_i$ where each $E_i$ is an $R$-algebra,  then $\mathcal X (E) = \cap \mathcal X (E_i)$.
\end{proposition}

We will use Proposition~\ref{LAZZEPI} under the following form.

\begin{proposition}\label{EPIPLAT} $(\mathcal X$)-rule  Let $R\to E$  be a flat epimorphism and $R \to F$ a ring morphism.
\begin{enumerate}

\item There is a factorization $R\to E \to F$ if and only if $\mathcal X(F) \subseteq \mathcal X(E)$.

\item  If $R\to F$ is a flat epimorphism,  then  $E \cong F$ if and only if $\mathcal X(F) =\mathcal X(E)$.
\end{enumerate} 
\end{proposition}

\begin{proof} (1) The ring morphism $\alpha:F\to F\otimes_RE$ is a flat epimorphism. If $\mathcal X(F)\subseteq\mathcal X(E)$, then the spectral morphism of $\alpha$ is surjective, because there is a surjective map $ \mathrm{Spec}(F\otimes_R E )\to \mathrm{Spec}(E)\times_{\mathrm{Spec}(R)} \mathrm{Spec}(F)$ \cite[Corollaire 3.2.7.1, p.235]{EGA}.
  It follows that $\alpha$ is a faithfully flat epimorphism, whence an isomorphism by \cite[Lemme 1.2, p.109]{L} and an implication is proved. Its converse is obvious. Now (2) can be proved by using (1). But it  is  also a consequence of Proposition~\ref{LAZZEPI}. 
\end{proof}

\begin{corollary}$(\mathcal{MCS}$)-rule Let $R\to E$ be a ring morphism. 
\begin{enumerate}
\item If $E = R_\Sigma$ where $\Sigma$ is a  mcs of $R$, then $\mathcal X (E) = \cap [\mathrm D(s) | s\in \Sigma]$.

\item If $E=\varinjlim R_{s_i}$, where $\{s_i\}$ is family of elements of $R$, then $E =R_\Sigma$, where $\Sigma$ is the mcs of $R$ generated by the family.
\end{enumerate}
\end{corollary}
\begin{proof} The proof is a consequence of the above rules.
\end{proof} 

\section{S-regular ideals and rings of sections}

If $I$ is an ideal of a ring $R$, then $\Gamma (\mathrm{D}(I),R)$ (or $\Gamma (\mathrm{D}(I))$) denotes the ring of sections of the scheme $\mathrm{Spec}(R)$ over the open subset $\mathrm D (I)$. All that we need to know is that $\Gamma(\mathrm D(R))=R$, $\Gamma(\emptyset)= 0$ and if $f:R\to S$ is a ring morphism, there is a commutative diagram, because ${}^af^{-1}(\mathrm{D}(I))= \mathrm{D}(IS)$: 

\centerline{\car R, S, {\Gamma (\mathrm{D}(I))}, {\Gamma (\mathrm{D}(IS))}}

We denote by $\mathrm{Ass}(R)$ the set of all (Bourbaki) prime ideals $P$ associated to the $R$-module $R$; that is, $P\in \mathrm{Min}(\mathrm{V}(0:r))$ for some $r\in R$. 
Recall that a ring morphism $f:R \to S$ is called {\it schematically dominant} if for each open subset $U$ of $\mathrm{Spec}(R)$, the map $\Gamma (U,R) \to \Gamma ({}^af^{-1}(U),S)$ is injective \cite[Proposition I.5.4.1]{EGA}. The first author proved that a flat ring morphism $f:R \to S$ is schematically dominant if and only if $\mathrm{Ass}(R)\subseteq\mathcal X (S)$ \cite[Proposition 52]{PS}.
  Clearly if $\mathrm{Min}(R)=\mathrm{Ass}(R)$ (for example, if $R$ is an integral domain) and $f$ is injective and flat, then $f$ is schematically dominant.
 
\begin{lemma}\label{1.10.5} A flat extension $R\subseteq S$ is schematically dominant.
 \end{lemma}
 \begin{proof} 
If $ P\in\mathrm{Ass}(R)$, there is some $a\in R$, such that $P\in \mathrm{Min}(\mathrm{V}(0:a))$. From $(0:_Sa)\cap R=0:a$, we deduce that $R/(0:a)\to S/(0:_Sa)$ is injective and then $P/(0:a)$ can be lifted up to a minimal prime ideal $Q/(0:_Sa)$. Hence $Q\in \mathrm{Ass}(S)$ is above $P$.
  \end{proof}
  
  Let $R \subseteq S$ be an extension and  an ideal   $I$  of $R$.
  Then $I$ is called $S$-{\it regular} if $IS=S$ \cite{KZ}. Note that $S$-regular ideals play a prominent role in \cite{KZ}. They are involved in certain questions. For example, if $f:R\to S$ is a ring morphism, the fiber at a prime ideal $P$ of $R$ is ${}^af^{-1}(P)$. This fiber is homeomorphic to the spectrum of the ring $S_P/PS_P$. Therefore, the fiber is empty if and only if $S_P=PS_P$, which means that $PR_P$ is $S_P$-regular. If $f$ is a flat epimorphism, the fiber at $P$ is empty if and only if $S = PS$ 
  \cite[Proposition 2.4, p.111]{L}.

\begin{remark}\label{RS} Let $f: R \hookrightarrow S$ be an extension. 

(1) An ideal $I$ of $R$ is $S$-regular if and only if $\mathcal X(S)\subseteq \mathrm D(I)$ \cite[Lemma 2.3]{PG}. Such an ideal $I$ is dense; that is, $0:I= 0$. 

(1(a)) $I$ is $S$-regular if and only if $\sqrt I$ is $S$-regular, because $\mathrm D(\sqrt IS)={}^af^{-1}(\mathrm D(\sqrt I))={}^af^{-1}(\mathrm D(I))= \mathrm D(IS)$. 

(1(b)) $I$ is $S$-regular if and only if  $I_P$ is $S_P$-regular for each $P\in \mathrm{Spec}(R)$. We need only to show that if the local condition holds, then $I$ is $S$-regular. Suppose that $IS \subset S$, then there is some prime ideal  $Q$ of $S$, such that $IS \subseteq Q$. If $P= Q\cap R$, then $Q_P$ is a prime ideal of $S_P$, such that $I_PS_P \subseteq Q_P$, a contradiction.

(2) If $I$ is $S$-regular, we have $\mathrm{Spec}(S)=\mathrm D(IS)={}^af^{-1}(\mathrm D(I))$, so that there is a factorization $R\to\Gamma(\mathrm D(I))\to S$. 
  If, in addition, $f$ is flat, then $f$ is schematically dominant (Lemma~\ref{1.10.5}); so that, we can consider that there is a tower of extensions $ R\subseteq\Gamma(\mathrm D(I))\subseteq S$. Moreover, $\mathrm D(I)$ is an open subset which is (topologically) dense in $\mathrm{Spec}(R)$, because a schematically dominant morphism is dominant \cite[Proposition I.5.4.3]{EGA},
  {\it i.e.} its spectral image is dense. The density follows from $\mathcal X(S) \subseteq \mathrm{D}(I)$.
 
This result holds if the extension is Pr\"ufer and then $R\to\Gamma(\mathrm D(I))$ is Pr\"ufer.
 
(3) We will use the following consequence of Proposition~\ref{LAZZEPI}.
If $I$ is an ideal of $R$, then $R\to\Gamma(\mathrm D(I))$ is a flat  epimorphism and $\mathcal X(\Gamma(\mathrm D(I)))=\mathrm D(I)$ if and only if $\mathrm D(I)$ is an affine open subset of $\mathrm{Spec}(R)$ (for example if $I$ is principal), in which case $\mathrm D(I) = \mathrm D(J)$ where $J$ is a finitely generated ideal. 
\end{remark}

We can say more after looking at the following recall adapted to ring morphisms (the reader is referred to \cite[Definition I.4.2.1, p.260]{EGA} for the definition of an open immersion of schemes). We will say that a ring morphism is an {\it open immersion} if the morphism of schemes associated is an open immersion.

\begin{proposition}\label{IM}  Let $f: R \to S$ be a ring morphism.

\begin{enumerate}
\item \cite[I.4.2.2]{EGA} $f$ is an open immersion if and only if $\mathrm{Spec}(S)\to\mathcal X (S)$ is a homeomorphism, $\mathcal X(S)$ is an open subset and the local morphisms of $f$ are isomorphisms. 

\item A flat epimorphism $R\to S$, such that $\mathcal X(S)$ is Zariski open is an open immersion.

\item \cite[Th\'eor\`eme 17.9.1, p.79]{EGAIV} $f$ is an open immersion if and only if $f$ is a flat epimorphism of finite presentation.

\item \cite[Theorem 1.1]{CR}  An injective flat epimorphism of finite type is of finite presentation, whence is an open immersion.
\end{enumerate} 
\end{proposition}
\begin{proof} We need only to prove (2) and using  (1).
 Since $f$ is a flat epimorphism, its spectral map is an homeomorphism  onto its image by \cite[Corollaire 2.2, p.111]{L} which is an open subset of the form $\mathrm D(I)$.
   Moreover, the local morphisms of the map are isomorphisms.
\end{proof}

\begin{proposition}\label{AFFINJ}  Let $R\subset S$ be an injective flat epimorphism of finite type. Then $\mathcal X (S)$  is an open affine subset  $\mathrm D (I)$, where  $I$ is a $S$-regular ideal   and  there is  an $R$-isomorphism   $\Gamma (\mathrm{D}(I))\cong S$,  $R\subset S$ is of finite presentation and $I$ is a dense ideal.
   
Conversely, if $\mathrm{D}(I)$ is an open affine subset, where $I $ is a finitely generated dense ideal, then $R\to \Gamma (\mathrm{D}(I))$ is an injective flat epimorphism, of finite type (presentation), such that $\mathcal X (\Gamma (\mathrm{D}(I)))= \mathrm{D}(I)$. 
   \end{proposition} 
   \begin{proof}    
To apply Proposition~\ref{LAZZEPI}, we need only to look at injective flat epimorphisms of finite type $R \to S$.  
     We know that such a ring morphism $f: R\subset S$ is of finite presentation  according to Proposition~\ref{IM}(4). 
  By the Chevalley Theorem, $\mathcal X(S)$ is a Zariski quasi-compact open subset of $\mathrm{Spec}(R)$, therefore of the form $\mathrm D(I)$, where $I$ is an ideal of $R$, of finite type. We have ${}^af^{-1}(\mathrm D(I))=\mathrm D(IS)=\mathrm{Spec}(S)$ because $\mathcal X(S)=\mathrm D(I)$, so that $IS= S$ and then $I$ is dense because it is  $S$-regular. 
   Moreover, $\Gamma(\mathrm{D}(I))\cong S$ by Proposition \ref{EPIPLAT}(2) because $\mathcal X(S)=\mathrm D(I)=\mathcal X (\Gamma (\mathrm{D}(I)))$.  
   
Assume that the hypotheses of the converse hold. Since the morphism $R \to \Gamma(\mathrm{D}(I))$ is an open immersion by Proposition~\ref{IM}, we get that $R\to\Gamma(\mathrm{D}(I))$ is of finite presentation. Moreover, $0:I=0$ (which is equivalent to $\mathrm{Ass}(R)\subseteq\mathrm{D}(I)$ \cite[Corollaire 1.14, p.93]{L}), so that $R\to \Gamma(\mathrm{D}(I))$ is injective,  by  \cite[Proposition 3.3, p.96]{L}.
  \end{proof}
 
  We note the following result:
\begin{proposition} \cite[Theorem 2.8, p.101, Theorem 2.6, p.100]{KZ} Let $R\subseteq S$ be an extension which is a flat epimorphism. Then the extension is Pr\"ufer if and only if for every finitely generated $S$-regular ideal $I$ of $R$,  the ring $R/I$ is arithmetical (resp.; $I$ is locally principal).
   \end{proposition} 
  \begin{proposition} 
  Let $R\subseteq  S$ be a flat epimorphism. Then,   
 $R\subseteq S$ is Pr\"ufer if and only if for each $P\in \mathrm{Spec}(R)$, the   set of  $S_P$-regular ideals of $R_P$ is a  chain.
  \end{proposition}
\begin{proof} According to \cite[Proposition 1.1(2)]{Pic 5}, the extension is Pr\"ufer if and only if $R_P \subseteq S_P$ is Manis for each $P\in \mathrm{Spec}(R)$ and equivalently $R_P \subseteq S_P$ is Pr\"ufer-Manis. The result follows from \cite[Theorem 3.5, p.190]{KZ}.
  \end{proof} 
  
We recall that the dominion of a ring morphism $f:R\to S$ is the subring $\mathrm{Dom}(f) = \{x\in S\mid x\otimes1=1\otimes x \, \, \mathrm{in} \, \, S\otimes_RS \}$ of $S$, which contains the subring $f(R)$. Actually,
  $\mathrm{Dom}(f)$ is the kernel of the morphism of $R$-modules $i_1-i_2:S \to S\otimes_RS$ where $i_1, i_2$ are the natural ring morphisms $S \to S\otimes_RS$.  
       
\begin{proposition}\label{DOM} If $f:R\to S$ is a flat morphism and $I$ an ideal of $R$, such that $\mathcal X_R(S) =\mathrm{D}(I)$, then 
 \begin{enumerate}
 
 \item$\Gamma (\mathrm{D}(I)) = \mathrm{Dom}(f)$ and $\Gamma (\mathrm{D}(I)) \to S$ is an injective flat morphism.
 
\item If in addition $f$ is a ring extension, $\widetilde R\subseteq\widehat R \subseteq\Gamma(\mathrm{D}(I))$, each of the extensions in $S$ being flat. In particular, if $\mathrm D(I)$ is affine, then $\widehat R=\Gamma(\mathrm{D}(I))$. 
 
\item If $g:R\to B$ is a flat morphism, setting $C:=S\otimes_RB$, then $\mathcal X_B(C)=\mathrm{D}(IB)$ and
 $\Gamma (\mathrm{D}(I))\otimes_RB \cong \Gamma (\mathrm{D}(IB))$.
 
\item If $P$ is a prime ideal of $R$, then $\Gamma(\mathrm{D}(I_P))= (\Gamma (\mathrm{D}(I)))_P$. In particular if $P\in \mathrm D(I)$, then $(\Gamma (\mathrm{D}(I)))_P = R_P$.
 
 \item  $\mathrm D(I )\subseteq \mathcal X(\Gamma (\mathrm D(I)))$.
 
\item If $I\Gamma(\mathrm D(I))=\Gamma(\mathrm D(I))$, then $\mathrm D(I )=\mathcal X(\Gamma(\mathrm D(I)))$, so that $\mathrm D(I)$ is an open affine subset if in addition $R\to\Gamma  (\mathrm D(I))$ is a flat epimorphism.
 \end{enumerate}
    \end{proposition} 
  \begin{proof}  (1) is a translation of \cite[Theorem 2.7]{PG}.  The flatness of  $ \Gamma (\mathrm{D}(I)) \to S$ follows from \cite[Proposition 3.1 (2), p.112]{L}. 
  
(2) If $f$ is a ring extension, observe that $\widehat R\subseteq
\mathrm{Dom}(f)$, because $R\to\widehat R$ is an epimorphism and then $y\otimes 1 = 1\otimes y$ for each $y \in \widehat R$ \cite[Lemme 1.0, p.108]{L}. The flatness of the extensions $\widetilde R,\widehat R\subseteq S$ result from \cite[Proposition 3.1(2), p.112]{L}. 
   At last, if $\mathrm D(I)$ is affine, then $R\to\Gamma (\mathrm{D}(I))$ is a flat epimorphism by Remark \ref{RS}(2), so that $\widehat R=\Gamma (\mathrm{D}(I))$ . 

(3) Because $\mathrm{Spec}(C)\to\mathrm{Spec}(B)\times_{\mathrm{Spec}(R)} \mathrm{Spec}(S)$ is a surjective map \cite[Corollaire 3.2.7.1, p.235]{EGA}, we have $\mathcal X_B(C)={}^ag^{-1}(\mathrm D(I))=\mathrm D(IB)$. To conclude use (1) and the fact that a kernel tensorised  by $B$, which is flat over $R$,  is the kernel of the tensorised map.

 (4) is gotten by taking $B=R_P$ in  (3).

(5) According to (4), an element $P$ of $\mathrm D(I))$ is such that $(\Gamma (\mathrm{D}(I)))_P = R_P$. It follows that there is a prime ideal  $Q$ of $\Gamma (D(I))$  lying over $P$.

 (6) holds because $I$ is $\Gamma(\mathrm D(I))$-regular.
 \end{proof}      
We can apply the above result in the following three contexts, when $I= (r_1,\ldots, r_n)$ is an ideal of finite type of $R$ (the hypothesis of this result entails that $\sqrt I=\sqrt J$, where $J$ is an ideal of finite type). We can suppose that $\mathrm D(I)\neq\emptyset$ and that the set $\{\mathrm D(r_1),\ldots,\mathrm D(r_n)\}$ is an antichain; so that, the $r_i~'s$ cannot be nilpotent. The first sample is certainly the most interesting, because when $I= Rr$, we recover that $\Gamma(\mathrm D(r))= R_r$.
  
(1) We can  consider the flat ring morphism $\varphi:R\to R_{r_1}\times\dots \times R_{r_n} := S_I$, which is such that $\mathcal X (S_I) = \mathrm D(I)$. Actually, $\varphi$ is of finite presentation \cite[Proposition 6.3.11, p.306]{EGA} and its local morphisms are isomorphisms. But $\varphi$ may not be a flat epimorphism, when it is not an $i$-morphism. 
  
In case $\{\mathrm D(r_i)\}$ defines a partition on $\mathrm D(I)$, $\varphi$ is an $i$-morphism, whence a flat epimorphism. In this case, $\Gamma (\mathrm{D}(I)) =\prod R_{r_i}$.
  
(2) Let $F_I:=R[X_1,\ldots,X_n]/(r_1X_1+\cdots+r_nX_n-1)$ be the forcing $R$-algebra with structural morphism $f_I$, associated to a finitely generated ideal $I= (r_1,\ldots ,r_n)$ (it would be more correct to write: associated to the sequence $\{r_1,\ldots ,r_n\}$). This ring is not zero, for otherwise 
   $1=(r_1X_1+\cdots+r_nX_n-1)P(X_1,\ldots,X_n)$, for some $P(X_1,\ldots,X_n)\in R[X_1,\ldots,X_n]$, implies that $1-(r_1X_1+\cdots+ r_nX_n)\in\mathrm U(R[X_1,\ldots,X_n])$, so that $r_1X_1+\cdots+r_nX_n$ would be nilpotent and then also the $r_i$'s. Then $I$ is $F_I$-regular and  for every ring morphism $R \to S$ for which $I$ is $S$-regular, there is a factorization $R \to F_I \to S$. But $F_I \to S$ does not  need to be unique.
       According to \cite[Theorem 2.7 and Remark 2.8(1)]{PG}, the ring morphism $f_I$ is flat, $\mathcal X(F_I)=\mathrm{D}(I)$ and $\Gamma(\mathrm{D}(I))= \mathrm{Dom}(f_I)$. Moreover, $\Gamma(\mathrm{D}(I))\to F_I$ is an injective flat ring morphism.
 
(3) The first author introduced in \cite{PicC} the following construction that we adapt to the present context. Let $R$ be a ring and $I$ an ideal of $R$. Denoting by $C(p(X))$ the content of a polynomial $p(X)\in R[X]$, we consider the mcs $\Sigma:=\{p(X)\in R[X]\mid\mathrm D(I)\subseteq\mathrm D(C(p(X)))\}$. Setting $R(\mathrm D(I))=:R[X]_\Sigma$, we get a flat morphism $R \to R(\mathrm D(I))$, such that $\mathcal X(R(\mathrm D(I))) =\mathrm D(I)$.  

 \section{Integral closures  as  intersections} 

We start by giving some results that do not seem to have been observed. They are consequences of a paper by P. Samuel \cite{SAM}.  
 Let $v$ be a valuation on a ring $R$. Following  \cite{KZ}, we denote by  $A_v$ the valuation ring of $v$.

\begin{lemma}\label{1.1} \cite[Th\'eor\`eme 1(d)]{SAM} An extension $R\subset S$, such that $S\setminus R$ is multiplicatively closed, is integrally closed. For example, $R\subseteq S$ is integrally closed if there is some valuation $v$ on $S$ such that $R= A_v$,  the valuation ring of $v$.  
\end{lemma} 

We will use the next result.

\begin{lemma}\label{MAX} \cite[Theorem 4]{SAM} Let $R\subseteq S$ be an extension and $P$ a prime ideal of $R$. Due to Zorn Lemma, there is a maximal pair $(R',P')$ dominating $(R,P)$ and $S\setminus R'$ is a mcs.
\end{lemma} 
 
 A Manis valuation $v$  on a ring $S$ is a valuation such that $v:S \to \Gamma_v \cup {\infty}$ is surjective, where $\Gamma_v$ is the value group of $v$.
 
A ring extension $R\subseteq S$ is called Manis if $R = A_v$ for some Manis valuation $v$ on $S$. Pr\"ufer-Manis extensions are defined as Pr\"ufer extensions $R\subset S$ such that there is some Manis valuation $v$ on $S$ such that $A_v= R$ \cite[Definition  1, p. 58]{KZ}.

By \cite[Theorem 3.5, p.190]{KZ}, a flat epimorphism $R\subseteq S$ is Pr\"ufer-Manis if and only if the set of all $S$-regular ideals of $R$ is a chain. 

\begin{lemma} \cite[Theorem 3.3, p.187, Theorem 3.1, p.187, Proposition 5.1(iii), p. 46-47]{KZ}. The following statements are equivalent for an extension $R\subseteq S$.
\begin{enumerate}

\item  $R\subseteq S$ is Pr\"ufer-Manis. 

\item  $S\setminus T$ is a mcs for each $T\in [R,S[$. 

\item $R\subseteq S$ is integrally closed and chained.

\item $R\subseteq S$ is Pr\"ufer and $S\setminus R$ is a mcs.
\end{enumerate}
\end{lemma}

If the above condition (3) holds for an FCP extension, then $R\subseteq S$ has FIP.

By \cite{KZ}, we know that for a Pr\"ufer extension $R\subset S$ and $P\in  \mathrm{Supp}(S/R)$, the subset $S_P\setminus R_P$ is multiplicatively closed. Also \cite[Proposition 5.1(iii), p. 46-47]{KZ} shows that if $U\in [R,S]$ and $S\setminus U$ is a mcs, then $U\subset S$ is Pr\"ufer-Manis.

 \begin{corollary} \label{PM} A minimal extension $R\subset S$ is a flat epimorphism if and only if it is Pr\"ufer and if and only if it is Pr\"ufer-Manis.
 \end{corollary} 
\begin{proof} The extension is a flat epimorphism if and only if it is integrally closed. To complete the proof it is enough to use \cite[Proposition 3.1]{FO} which states that $S\setminus R$ is a mcs when $R\subset S$ is a flat epimorphism.
 \end{proof}
 
\begin{corollary}\label{PMFCP} An FCP Pr\"ufer extension has FIP and is a tower of finitely many Pr\"ufer-Manis minimal extensions. 
  \end{corollary}

We will need the two following results. They generalize known results about the integral closure of an integral domain, which is the intersection of valuation rings. 

\begin{lemma}\label{prelim} Let $R\subset S$ be an extension and $x\in S\setminus\overline R$. Then there is some $T\in [R,S]$, such that $T\subset S$ is Pr\"ufer (respectively, Pr\"ufer-Manis) and $x$ is not integral over $T$, and then $x\notin T$.
\end{lemma} 

\begin{proof} It is enough to mimic the first part of the proof of \cite[Th\'eor\`eme 8]{SAM}. More precisely, let $T$ be a maximal element of the $\cup$-inductive set $\{U\in[R,S]\mid x\notin\overline U^S\} $. We intend to show that $T\subset S$ is Pr\"ufer-Manis. 
In view of the above results, we need only to show that any $V \in [T,S]$ with $V\neq S$ is such that $S\setminus V$ is multiplicatively closed and then such that $T\subset S$ is integrally closed. Now replace $A$ with $V$ in the second paragraph of the proof of \cite[Th\'eor\`eme 8]{SAM}, and the result follows.
 \end{proof}
 
\begin{theorem}\label{Intclos} Let $R\subset S$ be an extension, then $\overline R^S$ is the intersection of all $T\in[R,S]$ such that $T\subset S$ is Pr\"ufer (resp. Pr\"ufer-Manis) and also the intersection of all $U\in [R,S]$, such that $S\setminus U$ is a mcs.
 \end{theorem}
 \begin{proof} The second result is \cite[Th\'eor\`eme 8]{SAM}. Now Lemma~\ref{prelim} shows that $\overline R^S$ contains the intersection of all $T\in[R,S]$ such that $T\subset S$ is Pr\"ufer  (resp. Pr\"ufer-Manis).
 
For the reverse inclusion, consider an element $x\in S$, integral over $R$. Then $T\subseteq T[x]$ is integral and a flat epimorphism for any $T\in[R,S]$ such that $T\subset S$ is Pr\"ufer. We deduce from \cite[Lemme 1.2, p. 109]{L}, that $T= T[x]$ and $x$ belongs to $T$.
 \end{proof}
 
 \begin{remark} As a consequence of the above Theorem, we get that an extension $R\subseteq S$ is quasi-Pr\"ufer if and only if the set of all $T\in [R,S]$, such that $T\subseteq S$ is Pr\"ufer, has a smallest element.
 \end{remark}
 
 \section{Avoidance lemmata}
 Some of the following  results are known in the context of  integral domains and valuation domains. We will use the frame of their proofs  but  shorter different argumentations. 
 Kostra proved the next Theorem \cite[Lemma 2 and Theorem 2]{KOS}, in case $S$ is a field. To prove it in our context, we follow the steps of his difficult proof by using Theorem~\ref{Intclos}. 
 
If $V\subseteq S$ is Pr\"ufer-Manis, $S$ is endowed with a valuation $v: S \to \Gamma_v\cup \{\infty\}$, which is surjective. 
 There is no need to consider invertible elements that may not exist but elements $x\in V$, such that $v(x)=0$, that is $x\notin P_v$, the center of $v$. Moreover, if $v(x)>0$, then $v(1+x)=0$.
 
 \begin{lemma}\label{COV} Let  $R,  T, V, V_1,\ldots, V_n $  be subrings of a ring $S$, where $n$ is a positive integer and such that $V\subseteq S$ and $V_i\subseteq S$ are Pr\"ufer-Manis for each $i\in\mathbb N_n$. Let $v_i$ be the valuation associated to $V_i$. Assume that there is some $b\in [T\cap (\cap_{i\in\mathbb N_n}V_i)]\setminus V$.  Then, there exists $c\in[T\cap (\cap_{i\in\mathbb N_n}V_i)]\setminus V$ such that $v_i(c)=0$ for any $i\in\mathbb N_n$. 

Moreover, for any $W\in[R,S]$ such that $W\subseteq S$ is Pr\"ufer-Manis with $b\not\in W$, then $c\not\in W$.
\end{lemma}
\begin{proof}  We build by induction the sequence $\mathcal S:=\{b_i\}_{i=0}^n$ in the following way: set $b_0:=b$ and $b_k:=1+\prod_{i=0}^{k-1}b_i$ for any $k\in\mathbb N_n$. Then, $b_k\in T\cap(\cap_{i\in\mathbb N_n}V_i)$ for any $k\in\{0,\ldots,n\}$, so that $v_i(b_k)\geq 0$ for any $i\in\mathbb N_n$ and any $k\in\{0,\ldots,n\}$.

If $b_l=b_k$ for some $k\neq l$, assume that $k>l$. Then, $b_k=1+b_k\prod_{j=0,j\neq l}^{k-1}b_j$, giving that $b_k(1-\prod_{j=0,j\neq l}^{k-1}b_j)=1$, 
 so that $v_i(b_k)+v_i(1-\prod_{j=0,j\neq l}^{k-1}b_j)=v_i(1)\ (*)$, with $b_k$ and $1-\prod_{j=0,j\neq l}^{k-1}b_j$ both in $V_i$ for any $i\in\mathbb N_n$. It follows that $v_i(b_k)\geq 0$ and $v_i(1-\prod_{j=0,j\neq l}^{k-1}b_j)\geq 0$. As $v_i(1)=0$, $(*)$ leads to $v_i(b_k)=0$ for any $i\in\mathbb N_n$ and the proof of the Lemma is gotten for $b_k$.

Assume now that $b_j\neq b_k$ for any $k,j\in\{0,\ldots,n\},\ k\neq j$, so that $|\mathcal S|=n+1$.

We claim that for any $i\in\mathbb N_n$, there is at most one $b_{i_j}\in \mathcal S$ such that $v_i(b_{i_j})>0\ (**)$. 

If $v_i(b_k)=0$ for any $k\in\{0,\ldots,n\}$, then $(**)$ holds. Otherwise, let $j_0$ be the least integer of $\{0,\ldots,n\}$ such that $v_i(b_{j_0})\neq 0$, that is $v_i(b_{j_0})>0$. It follows that for any $k\geq j_0$, we have $v_i(\prod_{j=0}^k b_j)>0$, so that $v_i(b_{k+1})=v_i(1+\prod_{j=0}^kb_j)= v_i(1)=0$. Since $v_i(b_k)=0$ for any $k<j_0$, we get that $v_i(b_k)\neq 0$ if and only if $k=j_0$. Then $(**)$ holds. Hence, $|\{b_j\in\mathcal S\mid\exists i\in\mathbb N_n$ such that $v_i(b_j)\neq 0\}|\leq n<|\mathcal S|$. It follows that there exists some $c:=b_k\in[T\cap (\cap_{i\in\mathbb N_n}V_i)]$ such that $v_i(c)=0$ for any $i\in\mathbb N_n$.

It remains to show that $c\not\in V$. We prove by induction on $j\in\{0,\ldots,k\}$ that $b_j\not\in V$ for any $j\in\{0,\ldots,k\}$. This is satisfied for $j=0$ since $b_0=b$. Assume that $b_j\not\in V$ for any $j\in\{0,\ldots,l\}$ where $l<k$. But $S\setminus V$ is a mcs, so that $\prod_{j=0}^lb_j\not\in V$, which implies that $b_{l+1}\not\in V$ and then $c=b_k=1+\prod_{j=0}^{k-1}b_j\not\in V$.

Now, let $W\in[R,S]$ be such that $W\subseteq S$ is Pr\"ufer-Manis with $b\not\in W$. We follow the proof of \cite[Remark, page 173]{KOS}. We consider the previous sequence $\{b_j\}$ with $b_0:=b$ and  $b_j=1+\prod_{i=0,}^{j-1}b_i$. We still have $c\in T\cap(\cap_{i\in\mathbb N_n}V_i)$. Obviously, since $b\not\in W$, so is any $b_i$, and then $b_k=c\not\in W$ because $S\setminus W$ is an mcs.
\end{proof}

\begin{theorem}\label{COV 1} Let $R,B_1,\ldots,B_n$  be subrings of a ring $S$, where $n$ is a positive integer, $n>1$. If the $B_i$s are integrally closed in $S$, except at most two of them, and $R\subseteq B_1\cup \cdots\cup B_n$, then $R$ is contained in some of the subrings $B_i$.
\end{theorem}
\begin{proof} First, we may remark that $R\subseteq B_1\cup B_2$ implies that $R$ is contained in one of the subrings $B_1,B_2$ by an obvious property of additive subgroups. So, we may assume that $n\geq 3$ with $B_i$ integrally closed in $S$ for any $i\geq 3$. There is no harm to assume that $n$ is the least integer such that $R\subseteq B_1\cup \cdots\cup B_n$, that is $R\not\subseteq\cup_{i\in\mathbb N_n, i\neq j}B_i$ for each  $j\in\mathbb N_n\ (*)$. 

To prove the Theorem, it is enough to show that if $R$ is not contained in any of the subrings $B_i$, we get a contradiction, that is $R\not\subseteq B_1\cup\cdots\cup B_n$, or equivalently, there exists some $x\in R\setminus( B_1\cup \cdots\cup B_n)$.  This $x$ is gotten after  five steps. 

{\bf Step 1.} Assume that $R\subseteq B_1\cup\cdots\cup B_n$ with $R$ not contained in any of the subrings $B_i$. According to $(*)$, for any $j\in\mathbb N_n$, there exists $a_j\in(R\cap B_j)\setminus(\cup_{i\in\mathbb N_n, i\neq j}B_i)$.

Fix some $i\in\mathbb N_n,\ i\neq j,\ i>2$. Since $B_i\subseteq S$ is integrally closed, by Theorem \ref{Intclos}, there exists a family $\{V_{k,i}\}\subseteq [B_i,S]$ such that $V_{k,i}\subseteq S$ is Pr\"ufer-Manis, with $B_i=\cap V_{k,i}$. Let $v_{k,i}$ be the Manis valuation associated to $V_{k,i}$. As $a_j\not\in B_i$, there exists some $V_{j,i}$ such that $a_j\not\in V_{j,i}$.  Moreover, $a_j\in V_{k,j}$ for any $k$ if $j\geq 2$.

Set $\mathcal M:=\{V_{j,i}\mid i>2,i\neq j\}$. Then, $B_3\cup\cdots\cup B_n\subseteq\cup_{i>2,i\neq j}V_{j,i}=\cup_{V_{j,i}\in\mathcal M}V_{j,i}$. For each $a_k$, set $\mathcal M^{(k)}:=\{V_{j,i}\in\mathcal M\mid a_k\in V_{j,i}\}$ so that $V_{k,j}\in\mathcal M^{(j)}$ for any $k$ if $j\geq 2$.
 
{\bf Step 2. }
If $\mathcal M^{(k)}\neq\emptyset$, then $a_k\in[R\cap(\cap_{V_{j,i}\in\mathcal M^{(k)}}V_{j,i})]\setminus V_{k,i}$. It follows from Lemma \ref{COV} that there exists $c_k\in R$ such that $v_{j,i}(c_k)=0$ for any $V_{j,i}\in\mathcal M^{(k)}$ and $c_k\not\in V_{k,i}$. In particular, $c_k\in V_{j,i}$ for any $V_{j,i}\in\mathcal M^{(k)}$.

If $\mathcal M^{(k)}=\emptyset$, set $c_k:=a_k\in R$. Since $a_k\not\in V_{j,i}$ for any $V_{j,i}\in\mathcal M$, it follows that $c_k\not\in V_{j,i}$ for any $V_{j,i}\in\mathcal M$.

{\bf Step 3.}
Set $d_0:=\prod_{k=1}^nc_k$. Then, $d_0\in R$. We claim that $d_0\not\in V$, for any $V\in \mathcal M$. Let $V\in\mathcal M$. Then, there exist $i_0,j_0,\ i_0>2,i_0\neq j_0$ such that $V=V_{j_0,i_0}$, so that  $a_{j_0}\not\in V$. Whatever is $\mathcal M^{(j_0)}$, we have that $c_{j_0}\not\in V$. It is obvious if $\mathcal M^{(j_0)}\neq\emptyset$. If $\mathcal M^{(j_0)}=\emptyset$, then, $c_{j_0}\not\in V_{j,i}$ for any $V_{j,i}\in\mathcal M$. In particular, $c_{j_0}\not\in V$. It follows that $v_{j_0,i_0}(c_{j_0})<0$. 

Consider $c_k$ for some $k\neq j_0$. If $c_k\not\in V$, then $v_{j_0,i_0}(c_k)<0$. If $c_k\in V$, we cannot have $\mathcal M^{(k)}=\emptyset$, so that $\mathcal M^{(k)}\neq\emptyset$. If $V\in\mathcal M^{(k)}$, then, $v_{j_0,i_0}(c_k)=0$ and $a_k\in V_{j_0,i_0}$. If $V\not\in\mathcal M^{(k)}$, then $a_k\not\in V_{j_0,i_0}$ and $c_k\not\in V_{j_0,i_0}$ by Lemma \ref{COV}, which leads to $v_{j_0,i_0}(c_k)<0$. In any case $v_{j_0,i_0}(c_k)\leq 0$.

To conclude $v_{j_0,i_0}(d_0)=\sum_{k=1}^nv_{j_0,i_0}(c_k)\leq v_{j_0,i_0}(c_{j_0})<0$. This implies that $d_0\not\in V$ for any $V\in\mathcal M$, and then $d_0\not\in B_3\cup\cdots\cup B_n$. 

Set $\mathcal M_0:=\{V_{1,i}\mid i>2\}\cup\{V_{2,i}\mid i>2\}\cup\{V_{3,i}\mid i>3\}$, with $\{V_{3,i}\mid i>3\}=\emptyset$ if $n=3$. Obviously, $\mathcal M_0\subseteq\mathcal M$, so that $d_0\not\in V_{j,i}$ for any 
 $V_{j,i}\in \mathcal M_0$. 
 
Let $t_1,t_2\in\mathbb N,\ t_1\neq t_2$. We claim that $v_{j,i}(d_0^{t_1})\neq v_{j,i}(d_0^{t_2})$ for any $V_{j,i}\in\mathcal M_0$. Assume that $t_1>t_2$ and set $t:=t_1-t_2$, that is $t_1=t+t_2$. It follows that $d_0^{t_1}=d_0^{t_2}d_0^t$, so that $v_{j,i}(d_0^{t_1})= v_{j,i}(d_0^{t_2})+v_{j,i}(d_0^t)$. Now, $v_{j,i}(d_0^{t_1})=v_{j,i}(d_0^{t_2})$ implies $v_{j,i}(d_0^t)=0$, that is $d_0^t\in V_{j,i}$. But $V_{j,i}\subseteq S$ is Pr\"ufer-Manis, and then integrally closed, so that $d_0\in V_{j,i}$, a contradiction. Then, $v_{j,i}(d_0^{t_1})\neq v_{j,i}(d_0^{t_2})$. 
 
Let $l\in\{1,2,3\}$ and consider the corresponding $a_l$ defined at the beginning of the proof. Then, there exists at most one $t_{j,i,l}\in\mathbb N$ such that $v_{j,i}(a_l)=v_{j,i}(d_0^{t_{j,i,l}})$. If there does not exist such $t_{j,i,l}$, we have $v_{j,i}(a_l)\neq v_{j,i}(d_0)$. In this case, set $t_{j,i,l}=1$. It follows that in any case and for any $t>t_{j,i,l}$, we have $v_{j,i}(a_l)\neq v_{j,i}(d_0^t)$. Let 
 
\centerline{$t_0:=\sup\{1+t_{j,i,l} \mid j,l\in\{1,2,3\}, \ i \in \{3,\ldots,n\}, i>j\}.$} 
  
 Then, $v(d_0^{t_0})\neq v(a_l)$ for any $V\in\mathcal M_0$ and any $l\in\{1,2,3\}$. 
 
 {\bf Step 4.}
 Set $d:=d_0^{t_0}$. Then, $v(d)\neq v(a_l)$ for any $V\in \mathcal M_0\ (**)$. Moreover, for any $V\in\mathcal M$, we have $d_0\not\in V$, which  implies $d\not\in V$  since $V\subset S$ is integrally closed. In particular, $ d\not\in B_3\cup\cdots\cup B_n$, but $d_0\in R$ implies $d\in R\subseteq B_1\cup B_2\cup B_3\cup\cdots\cup B_n$, so that $d\in B_1\cup B_2$. Now, $B_1\cup B_2=(B_1\cap B_2)\cup(B_1\setminus  B_2)\cup(B_2\setminus B_1)$. 
 
{\bf  Step 5. }
 We are going to consider the three possible cases for $d$.
 
 (1) $d\in B_1\cap B_2$.
 
Set $x:=a_3+d\in R$. Since $a_3\in(R\cap B_3)\setminus(\cup_{i\in\mathbb N_n, i\neq 3}B_i)$, we have $a_3\not\in B_1\cup B_2$, so that $x\not\in B_1 \cup B_2$. Moreover, $d\not\in B_3$, which implies that $x\not\in B_3$. Let $i>3$.
 
If $v_{3,i}(d)<v_{3,i}(a_3)$, then $v_{3,i}(x)=v_{3,i}(a_3+d)=v_{3,i}(d)<0$ because $d\not\in V_{3,i}$. Then, $x\not\in B_i$. 
 
If $v_{3,i}(d)\geq v_{3,i}(a_3)$, then $v_{3,i}(d)>v_{3,i}(a_3)$ by $(**)$, so that $v_{3,i}(x)=v_{3,i}(a_3)<0$ because $a_3\not\in V_{3,i}$. Then, $x\not\in B_i$. 

It follows that $x\not\in B_1\cup B_2\cup B_3\cup\cdots\cup B_n$, a contradiction. 
 
 (2)  $d\in B_1\setminus  B_2$.

Set $x:=a_2+d\in R$. Since $a_2\not\in B_1$, we have $x\not\in B_1$ and since $a_2\in B_2$, this implies that $x\not\in B_2$, so that $x\not\in B_1\cup B_2$. Let $i>2$.
 
If $v_{2,i}(d)<v_{2,i}(a_2)$, then $v_{2,i}(x)=v_{2,i}(d)<0$ because $d\not\in V_{2,i}$. Then, $x\not\in B_i$.
 
If $v_{2,i}(d)\geq v_{2,i}(a_2)$, then $v_{2,i}(d)>v_{2,i}(a_2)$ by $(**)$, so that $v_{2,i}(x)=v_{2,i}(a_2)<0$ because $a_2\not\in V_{2,i}$. Then, $x\not\in B_i$.

It follows that $x\not\in B_1\cup B_2\cup B_3\cup\cdots\cup B_n$, a contradiction. 
 
 (3)  $d\in B_2\setminus  B_1$.
 
 The proof is similar as in (2) by changing $B_1$ and $B_2$.
 
To conclude, we get a contradiction in any case, so that there exists some $i$ such that $R\subseteq B_i$.
\end{proof}

\begin{proposition}\label{INTER}
 Let $R\subseteq S$ be a Pr\"ufer extension  and $U, B_1,\ldots, B_n \in [R,S]$  such that $B_1\cap \cdots \cap B_n \subset U$ and $U\subseteq S$ is Pr\"ufer-Manis. Then there is some $i$ such that $B_i \subseteq U$. 
\end{proposition}

\begin{proof} Actually, this result is given under an equivalent form in \cite[Theorem 1.4, p.4]{KZII}. 
 Let $B_1,\ldots, B_n \in [R,S]$ be such that $B_1\cap \cdots \cap B_n \subset U$. Then, we have $U = U(B_1\cap \cdots \cap B_n) = UB_1\cap \cdots \cap UB_n$ by \cite[Theorem 1.4(4), p.86-87]{KZ}. Now $[U,S]$ is  a chain \cite[Theorem 3.1, p. 187]{KZ}; so that, $U=UB_i$ for some $i$ and then $B_i\subseteq  U$.
  \end{proof}
  
Note that if the extension $R\subseteq S$ is Pr\"ufer-Manis, so is $U\subseteq S$ for any $U\in[R,S]$ \cite[Corollary 3.2, P. 187]{KZ}. 

Gotlieb proved the following result  for a ring extension $R\subset K$ where $R$ is an integral domain with quotient field $K$ \cite[Theorem 6]{GOT}.
   
  \begin{theorem}
Let $R\subset S$ be an extension and $T,T_1,\ldots,T_n\in[ R,S]$, such that  $T=R_\Sigma$, where $\Sigma$ is a mcs of $R$ and $R\to T_i$ is a flat epimorphism for $i=1,\ldots,n$ such that $T\subseteq T_1\cup\ldots\cup T_n$. Then $T$ is contained in some $T_i$.
   \end{theorem}
\begin{proof} Assume that $T$ is not contained in any $T_i$. By the $(\mathcal X$)-rule, there are prime ideals $P_i\in\mathcal X(T_i)\setminus\mathcal X(T)$ and $\mathcal X(T)=\{P\in\mathrm{Spec}(R)\mid P\cap\Sigma=\emptyset\}$. We set $I:=P_1\cap\cdots\cap P_n$. We deduce from $T\subseteq T_1\cup \ldots\cup T_n$ that $I\cap\Sigma=\emptyset$. There exists some prime ideal $P$ of $R$ such that $I\subseteq P$ and $P\cap\Sigma=\emptyset$. Then some $P_i$ is contained in $P$: so that, $P_i\in\mathcal X(T)$. Hence we get a contradiction.
  \end{proof}
 
\section{Pullback results}
  
  Consider the following pullback diagram  (D) in the category of commutative unital rings:  \Car R, {i}, S,f,g,V,{j},K
  where $i$ and $j$ are ring extensions.   It can be considered as a composite of the two diagrams:
  
   \centerline {$\mathrm{Ker} (D)$:\car R, S, {R/\mathrm{Ker}(f)}, {S/\mathrm{Ker}(g)}   and  $\mathrm{Im} (D)$:\car {f(R)},  {g(S)}, V,K}
  
   The first diagram is a pullback  because  $\mathrm{Ker}(f) =\mathrm{Ker}(g)$ thanks to the pullback diagram  (D).  It follows that $R = f(R)\times_{g(S)}S$.
   
It is of the form \car A,B,{A/I},{B/I} where $I$ is an ideal shared by the rings $A$ and $B$. We recall that in this case, $A\subseteq B$ is Pr\"ufer if and only if $A/I\subseteq B/I$ is Pr\"ufer (\cite[Proposition 5.8, p.52]{KZ}). 

It is easy to prove that the second diagram is a pullback   and is such that $f(R)$ is isomorphic to $V\cap g(S)$.

Recall that a ring $R$ is called {\it semi-hereditary} if each of its finitely generated ideals is a projective $R$-module. 
 
Olivier proved that an extension of rings $R\subset S$ is integrally closed if and only if there is a pullback diagram (D), where V is a semi-hereditary ring with an (absolutely flat) total quotient ring $K$ \cite[Corollary p.56]{O} or \cite[Th\'eor\`eme de Ker Chalon (2.1)]{OBIS}. 
   In this case, we call (DO) the diagram (D).
   Therefore, the Pr\"ufer property is not descended in pullbacks,  since $V\subset K$ is Pr\"ufer \cite[Theorem 2]{ENDO} and there are integrally closed extensions that are not Pr\"ufer.
  
On the other hand we have a pullback example provided by the following result.

\begin{proposition}\label{VALU}\cite[Theorem 6.8 and Theorem 6.10]{DPP2} If  $R$ is a local ring, an extension $R\subseteq S$ is Pr\"ufer if and only if there exists $P\in\mathrm{Spec}(R)$ such that $S=R_P$, $P=SP$ and $R/P$ is a valuation domain. Under these conditions, $S/P$ is the quotient field of $R/P$ and $P$ is a divided prime ideal of $R$ ({\it i.e.} comparable to each ideal of $R$).  
 In particular, $[R,S]$ is a chain.
\end{proposition}

\begin{proof} To complete the proof, observe that there is an order isomorphism  $[R,S] \to [R/P, S/P]$  given by $T\mapsto T/P$ for $T\in [R,S]$.
\end{proof} 

We now use Olivier's result to find a characterization of Pr\"ufer extensions. 
  
  \begin{theorem}\label{1.3.1} 
Let $R\subset S$ be an integrally closed extension and ($DO$) the pullback diagram where $V$ is semi-hereditary with total quotient ring $K$. Then, $R\subset S$ is Pr\"ufer if and only if $g(T)V\cap g(S)=g(T)$ for each $T\in[R,S]$ or equivalently, the following diagram ($D_T$) is a pullback, for each $T\in [R,S]$:

  \car T,S,{g(T)V},K
 
In that case, we have  $R= V\times_{g(T)V}T$ and $g(T)V\cong V\otimes_RT$.
  \end{theorem}
\begin{proof} We use the characterization of Pr\"ufer extensions by normal pairs and flat epimorphisms. Suppose that $(D_T)$ is a pullback. Since  an overring of a semi-hereditary ring is semi-hereditary (\cite[Corollary p.143]{CA}), Olivier's result implies that $T\subset S$ is integrally closed. Hence 
 $R\subset S$ is Pr\"ufer.
  We now prove the converse. Suppose that $R\subset S$ is Pr\"ufer. Then $R\subset T$ is a flat epimorphism. Tensoring the diagram (D) by $\otimes_RT$, we get another pullback diagram because the pullback $R$ is a kernel of a morphism of $R$-modules and $T$ is flat over $R$. We next identify the rings of the new pullback. We have clearly $T\cong R\otimes_RT$. Moreover we also have $S\otimes_RT\cong S$. This is a consequence of \cite[Satz 2.2 (d)]{ST} which states that if $M$ is a $T$-module and $R\to T$ an epimorphism, then $M\otimes_RT\cong M$ (an isomorphism of $T$-modules). We next show that $V\otimes_RT\cong g(T)V$. Consider the natural map $V\otimes_RT \to K$; its image is $g(T)V$. Then $V\to V\otimes_R T$ is a flat epimorphism deduced from $R\to T$ by the base change $R\to V$ and $V\to V\otimes_R T\to g(T)V$ is injective.
   It follows that $V\otimes_RT\to g(T)V$ is an isomorphism, because a flat epimorphism is essential by \cite[Lemme 1.2, p.109]{L}. Then we show that $K\otimes_RT \cong K$. We first observe that $ K\to K\otimes_RT$ is a flat epimorphism whose domain is an  absolutely flat ring. 
   This map is surjective. To see this, if $J$ is the kernel of the morphism, then $K/J\to K\otimes_RT$ is a faithfully flat epimorphism because $K/J$ is absolutely flat whence is an isomorphism by \cite[Lemme 1.2, p.109]{L}. Moreover, $V\to V\otimes_RT$ identifies to $V\to g(T)V$; whence is injective. As $V\to K$ is flat, the map  $K \to (V\otimes_R T)\otimes_V K \cong K\otimes_RT$ is injective, so that $K\cong K\otimes_RT$.
   
    Therefore, we have proved that there is a pullback diagram $(D_T)$. 
   To complete the proof, it is enough to consider $\mathrm{Im}(D_T)$, in which case the pullback condition on $T$ can be written $g(T)V \cap g(S) =  g(T)$. 
  \end{proof}
  
Nevertheless, we give some example of pullbacks where the ascent property holds.
  
\begin{proposition}\label{1.3.2} Let $I$ be an ideal of a ring $S$ and set $S'=S/I$. Denote by $\varphi$ the canonical map $S\to S/I$. Let $R'$ be a subring of $S'$ and  $R$  the pullback ring in the following diagram:
  \car R, S,{R'},{S'}
  Then $R\subset S$ is Pr\"ufer if and only of $R'\subset S'$ is Pr\"ufer.
   \end{proposition}
   \begin{proof} Clearly $I$ is an ideal shared by $R$ and $S$.
    Now, observe that $R'$ identifies to $(R+I)/I\cong R/(I\cap R)\cong R/I$.
  It is then enough to apply  \cite[Proposition 5.8, p.52]{KZ}.
     \end{proof}
 
  \section{The case of a local base ring}

When the base ring $R$ is local, we already gave a characterization of Pr\"ufer extensions in Proposition~\ref{VALU}.

\begin{definition}\label{DISTRI} An extension $R\subseteq S$ is called  {\em module-distributive} if $R\cap(X+Y) = (R\cap X) + ( R\cap Y)$ for each pair of $R$-submodules $(X, Y)$ of $S$ (cf. \cite[p.119]{KZ}). We say that $R\subseteq S$ is distributive if the lattice $[R,S]$ endowed with compositum  and intersection as laws is distributive. 
\end{definition} 

\cite[Theorem 5.4, p.121]{KZ} shows that an extension is $R\subseteq S$ module-distributive if and only if it is Pr\"ufer. 
As a consequence we get that the set of $R_P$-submodules of $S_P/R_P$ is a chain \cite[Corollary 2]{ERGO}, when the extension is Pr\"ufer. This gives a stronger result than that of Proposition~\ref{VALU}. Moreover, we see that a Pr\"ufer extension is both module-distributive and distributive. For the distributivity, use \cite[Proposition 5.18]{Pic 4} since a Pr\"ufer extension 
$R\subset S$ is arithmetical (that is $R_M\subset S_M$ is chained for any $M\in\mathrm{Max}(R)$). 

In order to get more results, we introduce the following considerations.
 
In view of \cite[Proposition 5.2, p.119]{KZ}, an ideal $I$ of a ring $R$ is  called {\it distributive} if $I+ (J\cap K)= (I+J)\cap (I+K)$ for all ideals $J, K$ of $R$. When $R$ is local, an ideal $I$ is distributive if and only if $I$ is comparable to each ideal (principal) ideal of $R$. In this case ($R$ is local), we will call $I$ a {\it strong divisor} if in addition $0:I=0$. The following result will be useful.
 
 \begin{proposition}\label{1.10.4}\cite[Example 5.1, p.119]{KZ} Let $R\subset S$ be a Pr\"ufer extension. An $S$-regular ideal $I$ of $R$ is distributive. In particular, such an ideal is  a strong divisor if $R$ is local.
  \end{proposition}
  
We can translate some results of \cite[Lemma 1.1, Corollary]{ERGO} as follows. Let $R$ be a ring and set $\Sigma:=\{\sigma\in R\setminus\mathrm Z(R)\mid R\sigma\, \mathrm{is \,distributive}\}$. Then $\Sigma$ is a saturated mcs of $R$. 
  Moreover, if $\mathcal T$ is a mcs of $R$, such that $\mathcal T\subseteq R\setminus \mathrm Z(R)$ and $R \subseteq R_{\mathcal T}$ is Pr\"ufer, then $\mathcal T \subseteq \Sigma$.

We next examine the local case. We may find in \cite[p.123]{KZ} the following  definition and result.

\begin{definition}\label{1.6.1} 
 A {\em strong divisor} $t$ of a local ring $R$ is an element $t$ of $R$, such that the ideal $Rt$ is a strong divisor. The set $\Delta(R)$ of all strong divisors of $R$ is a saturated mcs of $R$ and $\mathrm{U}(R)\subseteq \Delta (R)$.
\end{definition} 

We observe that for $t\in \Delta(R)$, the open subset $\mathrm{D}(t) = \{Q \in \mathrm{Spec}(R) \mid Q \subset Rt\}$.  

Recall that a  ring $R$ has a maximal Pr\"ufer extension $R \subseteq  \mathbb P (R):= {\widetilde R}^{\mathrm{Q}(R)}$ where ${\mathrm{Q}(R)}$ is the complete ring of quotients of $R$ (Utumi-Lambeck) \cite{KZ}.
 Then  $\mathbb P(R)$ is called  the {\it Pr\"ufer hull of $R$}. 

It is known that a Pr\"ufer extension $R\subset S$, where $R$ is local, is a QR-extension; that is, is such that each $T\in [R,S]$ verifies $T\cong R_\Sigma$ (an isomorphism of $R$-algebras) for some mcs $\Sigma$ of $R$. For more information, see Proposition~ \ref{1.14}. Next result refines this observation.

\begin{proposition}\label{1.7} Let $R$ be a local ring and  $R\subset S$ an extension.
  \begin{enumerate}
  
   \item $\mathbb P(R) = R_{\Delta (R)}$.
   
\item An extension $R\subset S$ is Pr\"ufer if and only if $S= R_{\Sigma}$, for some mcs $\Sigma \subseteq \Delta(R)$ and, if and only if $R\subseteq Rs$ ({\it i.e.} $s^{-1}$ exists and belongs to $R$), for each $s\in S\setminus R$. In this case $S\subseteq \mathrm{Tot}(R)$.
   \end{enumerate}
  \end{proposition}
  \begin{proof}  
  The proof is a consequence the following facts: $R_{\Delta(R)}$ is the Pr\"ufer hull of $R$. If $R\subset S$ is Pr\"ufer, there is some mcs $\Sigma\subseteq \Delta (R)$ such that $S=R_\Sigma$ in which case $R \subseteq R_\Sigma \subseteq R_{\Delta(R)}$ \cite[Remark 5.9, Proposition 5.10 p.123]{KZ}. The last assertion is \cite[Proposition 3.1]{DAV}.    
   \end{proof}
\begin{lemma}\label{1.9.1} Let $R\subset R_{\Sigma}:= S$ be an extension of  finite type, where $\Sigma$ is a mcs of $R$. Then there is some $x \in \Sigma$  such that $ S= R_x$. 
      \end{lemma}
\begin{proof} It follows from \cite[Theorem 1.1]{CR}, that $R\subset R_\Sigma:= S$ is of finite presentation, because it is an injective flat epimorphism of finite type. 
 Therefore, according to the ($\mathcal{MCS}$)-rule, $ \mathcal X(S) =\cap [\mathrm{D}(r) | r\in \Sigma]$ is an open subset of the patch topology (constructible topology) by the Chevalley Theorem and is even open because a flat morphism of finite presentation is open for the Zariski topology. 
     As the patch topology is compact and the sets $\mathrm{D}(r)$ for $r\in \Sigma $ are closed in this topology, we get that $\mathcal X(S)$ is the intersection of finitely many $\mathrm{D}(r_i)$ for $i= 1, \ldots,n$ with $r_i\in \Sigma$. Setting $x= r_1\cdots r_n$, we get that $\mathcal X(S) = \mathrm{D}(x)$ and then $S=R_x$  by  the ($\mathcal X$)-rule.
      \end{proof}
   
   The next result is now clear.
   
\begin{proposition}\label{LOCALPRUF} An extension $R\subset S$ of finite type over a local ring $R$ is Pr\"ufer if and only if there is some $s\in \Delta (R)$ such that $S= R_s$.   
   \end{proposition}
\begin{proof} Suppose that $R \subset S$ is Pr\"ufer then $S= R_\Sigma$ for some mcs $\Sigma \subseteq \Delta (R)$ (Proposition~\ref{1.7}). We deduce from Lemma~\ref{1.9.1} that $S = R_s$ for some $s\in \Delta (R)$. The converse is obvious.
        \end{proof}

   The following results will be useful.
     
\begin{proposition}\label{1.9.1.0} Let $R$ be a local ring and $x\in R$ a regular element. Then $x$ is a strong divisor if and only if $R\subseteq R_x$ is Pr\"ufer.
   \end{proposition}
\begin{proof} Proposition~\ref{LOCALPRUF} gives one implication. Suppose that $R\subset R_x$ is Pr\"ufer. From Proposition~\ref{LOCALPRUF} we deduce that $R_x= R_s$ for some strong divisor $s\in R$. It follows that $\sqrt{Rx}= \sqrt{Rs}$ and then $s^n = yx$ for some $n \in \mathbb N$ and $y\in R$. Therefore, $x$ is a strong divisor.
 \end{proof}
 
 \begin{example}\label{ARITH} Let $R$ be a local arithmetical ring. The set of all its ideals is a chain. It follows that  each regular element $x$ of $R$ is a strong divisor and then $R\subseteq R_x$ is Pr\"ufer.
 \end{example}
 
\begin{proposition}\label{1.9.2} Let $f: R \to R'$ be a faithfully flat ring morphism between local rings and $x\in R$. 
 If  $f(x)$ is a strong divisor, so is $x$.
 \end{proposition}
 \begin{proof} Observe that  $x $ is regular in $R$.
   To conclude, use Proposition~\ref{FAITHPRUF}, because $R'_{f(x)} =R_x\otimes_R R'$. 
 \end{proof} 
 
Let $R \subseteq S$ be an extension and $\Delta$ a mcs of $R$. The large quotient ring $R_{[\Delta]}$ of $R$ (in $S$) with respect to $\Delta$ is the set of all $x \in S$ such that there is some $s \in \Delta$ with $sx \in R$. In case $\Delta =R\setminus P$, where $P$ is a prime ideal of $R$, we set $R_{[P]}:= R_\Delta$. 
   
\begin{proposition}\label{1.8} Let $R$ be a local ring and $R\subset S$ a flat extension, then $\widetilde R = R_{[\Delta(R)]} = R_\Sigma$, where $\Sigma := \Delta(R) \cap \mathrm U (S)$. 
   \end{proposition}
\begin{proof} By Proposition \ref{1.7}, there is some multiplicatively closed subset $\Theta$ of $\Delta(R)$ such that $\widetilde R = R_\Theta$. We have clearly $\Theta\subseteq\mathrm U(S)$; so that $\Theta \subseteq \Sigma$. It follows that $R_\Theta \subseteq R_\Sigma$, while $R\subseteq R_\Sigma $ is Pr\"ufer and therefore  $R_\Theta = R_\Sigma$. 
        
   Now let $z\in R_{[\Delta(R)]}$, there is some $t\in \Delta(R)$ such that (*): $x=tz \in R$. Since $Rt $ is a strong divisor, $Rt$ and $Rx$ are comparable. 
 Moreover, since $R\subset S$ is flat, $t$ is also regular in $S$. 
    
 If $Rx\subseteq Rt$, then $x=at$, so that $z=a\in R$, because $t$ is regular. 
   
If $Rt\subseteq Rx$, then $t =bx$ and since $\Delta(R)$ is saturated, we get that $x\in \Delta(R)$ and $x$ is regular. We deduce from (*) that $bz=1$ in $S$. It follows that $z\in \mathrm U(S)$ and $z=b^{-1}$, with $b \in \Delta(R)\cap \mathrm U(S)$, so that $z\in R_\Sigma$. 
   
To conclude, we have $R_{[\Delta(R)]}\subseteq R_\Sigma$. As the reverse inclusion is obvious, we get finally that $R_{[\Delta(R)]}= R_\Sigma$.
 \end{proof}
  
If $Q$ is a prime ideal of a ring $R$, we denote by $Q^\downarrow$ its generization {\it i.e} $\{P\in \mathrm{Spec}(R) \mid P\subseteq Q\}$. The first author defined a {\it prime g-ideal} as a prime ideal  $Q$ such that $Q^\downarrow$ is an open subset of $\mathrm{Spec}(R)$ \cite{PGOLD}. $If$ $Q$ is a $g$-ideal of $R$, then $Q$ is a Goldman ideal of $R$; that is $R/P \subseteq \kappa(P)$ is of finite type as an algebra. \cite{PGOLD}.
     
\begin{proposition}\label{LPRUF} Let $s$ be a non-unit strong divisor of a local ring $R$ and $R\subset S:=R_s$ the Pr\"ufer extension associated. Then $P=  \cap [Rs^n \mid n\in \mathbb N\}$ is a prime $g$-ideal, $S=R_P$, $PS=P$ is a divided prime ideal of $R$ and $R/P$ is a valuation  domain  with quotient field $S/P$. We will denote by $P_s$ the ideal $P$.
     \end{proposition}

\begin{proof} There exists $P\in\mathrm{Spec}(R)$ such that $S=R_P$, $PS=P$ is a divided prime ideal of $R$ and $R/P$ is a valuation domain with quotient field $S/P$ according to Proposition~\ref{VALU}. 
    Note $<s>$ the saturated mcs generated by $s$ and set $I:=\cap [Rs^n\mid n\in\mathbb N]$. Remark that $<s>=\{us^n\mid n\in\mathbb N,\ u\in\mathrm U(R)\}$. Note that $P=R\setminus<s>$. We are aiming to show that $I=P$. We have $I\subseteq P$, because if not, there is some $x\in I\cap<s>$ and then $xy=s^p= bs^{p+1}$ for some $y,b\in R$ and $p\in\mathbb N$. Since $s$ is regular, it follows that $s$ is a unit, a contradiction. Now let $x \in P$ and suppose that $x\notin I$. Then $x \notin Rs^n$ for some positive integer $n$. Because $s^n$ is a strong divisor, we get $Rs^n \subseteq Rx$ and then $x$ belongs to $<s> $, a contradiction. Now $P$ is a prime $g$-ideal, because $P^\downarrow =\mathrm{D}(s)$ is an open subset. 
  \end{proof} 
If $R\subset S$ is a Pr\"ufer extension of finite type over a local ring, there is some $s\in\Delta(R)$ such that $S=R_s$ by Proposition~\ref{LOCALPRUF}.
  
\begin{remark} \label{NOETHPRUF} We use the notation of Proposition~\ref{LPRUF}.

It is easy to prove that $P= RsP$, because $s$ is regular and $P= \cap [Rs^n \mid n \in \mathbb N]$. Therefore, if $(R,M)$ is Noetherian and local and $s$ is not a unit, from $P=MP$ we deduce that $P=0$ and $R$ needs to be an integral domain, so that $R$ is a Noetherian valuation domain, that is a discrete valuation domain, and $S$ is the quotient field of $R$. Another consequence is that if $R$ is not an integral domain, the only strong divisors of $R$ are the units.
\end{remark} 

The next result is now clear.
\begin{proposition}\label{1.10} Let $R\subset S$ be an extension over a local ring. The set of rings $\mathcal F := \{ R_s \in [R,S] \mid s \in \Delta (R) \cap \mathrm U(S)\}$ is a chain and $\widetilde R$ is the set union of all elements of $\mathcal F$. It follows that $\widetilde R \hookrightarrow \mathrm{Tot}(R)$.
    \end{proposition}
   
\begin{definition} \label{1.11} We say that two ideals $I$ and $J$ of a ring $R$ are {\em equivalent} if $\sqrt I = \sqrt J$ (equivalently $\mathrm D(I) = \mathrm D (J)$). We also say that two elements $x, y$ of $R$ are equivalent if $\mathrm D(x) = \mathrm D(y)$ and we write $x\simeq y$. This condition is equivalent to $R_x\cong R_y$ and also to $\sqrt {Rx} =\sqrt {Ry}$. Note that if $x$ is a strong divisor and $x\simeq y$, then $y$ is a strong divisor because $\sqrt {Rx }= \sqrt {Ry}$ and the set of all strong divisors is a saturated mcs.
  \end{definition} 
  
\begin{remark} We reconsider the context of Proposition~\ref{LPRUF} and we set $\delta (R):= \Delta (R)\setminus \mathrm U(R)$. There is a surjective map $\delta (R) \to \{P_s \mid  s\in \delta (R)\} $, defined by $s\mapsto P_s$. Setting $\Delta \mathrm{Spec}(R):= \{P_s \mid s\in \delta (R)\}$, there is  therefore a bijective map 
 $(\delta (R)/\simeq) \to \Delta \mathrm{Spec}(R)$.
  
Then $\Delta\mathrm{Spec}(R)$ is a chain. It follows that the set intersection of all its elements is a prime ideal $\mathcal R$ that could be called the {\em strong radical} of the local ring $R$. Now, according to Proposition~\ref{1.10} 
   and the $(\mathcal {MCS})$-rule, $\mathcal X(\mathbb P(R))$ $=\cap[\mathrm D(s)\mid s\in \Delta(R)]=\cap[P_s^\downarrow\mid s\in \delta(R)]=\mathcal R^\downarrow$. It follows that $\mathbb P(R)=R_{\mathcal R}$. If $R\subset S$ is a ring extension, then $\widetilde R = \mathbb P(R)\cap S=R_{[\mathcal R]}$.
  
  We think that  the set $\Delta \mathrm{Spec}(R)$ deserves a deeper study, especially with respect to some classes of rings.
  \end{remark}
   
 \begin{proposition}\label{INVERT} Let $R \subset S$ be a Pr\"ufer extension, where $R$ is local and $I$ an  ideal of finite type of $R$. Then $I$ is $S$-regular if and only if $I =R\rho$ where $\rho$ is a strong divisor of $R$,  invertible in $S$.  
      \end{proposition}
   \begin{proof}   
   Assume that $I$ is $S$-regular.
   From $IS= S$ we deduce that $I$ is a principal ideal $R\rho$ by \cite[Theorem 1.13, p. 91 and Proposition 2.3, p.97]{KZ}, because $IS = S$ means that $I$ is $S$-regular and, $R$ being local, is $S$-invertible, whence principal of the form $I=R\rho$. An appeal to Proposition~\ref{1.10.4} yields that $R\rho$ is a strong divisor and $S\rho=S$ shows that $\rho$ is invertible in $S$.   The converse is obvious. 
 \end{proof}

\section{QR-extensions}    
   
We first give some notation and definitions for an extension $R\subset S$. 
  For $T\in [R,S]$, we set $\Sigma_T:=\mathrm{U}(T)\cap R$, which is a mcs of $R$ whose elements are regular and such that $R\subseteq R_{\Sigma_T} \subseteq T$.
  
A Pr\"ufer extension $R\subseteq S$ is called {\it Bezout}, if each finitely generated $S$-regular ideal of $R$ is principal \cite[Definition 1; Theorem 10.2, p.145]{KZ}.
 
Let $(R,M)$ be a local ring, then an extension $R\subseteq S$ is Bezout if and only if it is Pr\"ufer, and if and only if $(R,M)$ is Manis in $S$ \cite[Scholium 10.4 p.147]{KZ}. 
  
We call QR-extension any extension $R\subseteq S$ such that each $T\in [R,S]$ is of the form $T\cong R_\Sigma$ (an isomorphism of $R$-algebras)  for some mcs $\Sigma$ of $R$, in which case the elements of $\Sigma$ are regular, invertible in $S$ and $T = R_{\Sigma_T}$. It is easy to show that $R\subseteq S$ is a QR-extension if and only if the defining property holds for each $T\in [R,S]_{fg}$. Moreover, an extension $R\subset S$ is a QR-extension if and only if it is Pr\"ufer and each finitely generated $S$-regular ideal $I$ of $R$ satisfies $\sqrt I = \sqrt{Rx}$ for some $x\in R$ (which implies that $\mathrm D(I)=\mathrm D(x)$ is (special) affine) \cite[Proposition 4.15, p.116]{KZ}. 
 
A Pr\"ufer extension does not need to be a QR-extension: look at the example \cite[Section 4, Examples]{GIH}. 
  \begin{proposition} \label{1.27}A Bezout extension $R\subseteq S$ is a QR-extension.
   \end{proposition}
\begin{proof} We first observe that a subextension $R\subseteq T$ is Bezout. Then \cite[Proposition 10.16, p.152]{KZ} shows that $T= R_\Sigma$, for some mcs $\Sigma $ of $R$ and therefore the extension is QR.
\end{proof}

\begin{corollary} Each extension $R\subseteq S$ has a unique Bezout subextension  $R\subseteq T$, that contains any $T'\in [R,S]$, such that $R\subseteq T'$ is Bezout. Then $T$ is called the Bezout hull of $R$ and denoted here by $\beta (R)$.
\end{corollary}

\begin{proof} It is enough to use \cite[Theorem 10.14, p.151]{KZ}
\end{proof}

 Moreover, we have the next result. 
  
\begin{proposition}\label{1.14} Let $R \subset S$ be an extension where  projective $R$-modules of rank one are free. Then $R\subset S$ is Pr\"ufer if and only if it is a QR-extension, and if and only if $R\subset S$ is Bezout. If the above statements hold, then a finitely generated $S$-regular ideal $I$ of $R$ is of the form $I= R\rho$, where $\rho$ is a locally strong divisor.
    \end{proposition}
\begin{proof} The first equivalence is \cite[Proposition 4.16 p.116]{KZ}. The second is a consequence of \cite[Proposition 2.3, p.97]{KZ} because under the hypotheses on $R$, a Pr\"ufer extension is Bezout and the converse holds for an arbitrary ring $R$. The last statement is a consequence of Proposition~\ref{INVERT}.
    \end{proof}
 
The condition on projective modules that are involved in this paper are either $R$ is semilocal or a Nagata ring $A(X)$ \cite{FT}. In particular we recover Proposition~\ref{VALU} in case $R$ is a local ring.
 
We will need an extension of the notion of strong divisors. A regular element of a ring $R$ is called a {\it locally strong divisor} (shorten in $lsd$) if $R\subseteq R_x$ is Pr\"ufer. In order to justify this definition, we recall that an extension $R\subseteq S$ is Pr\"ufer if and only if all its localizations by a prime ideal of $R$ are Pr\"ufer. Hence if $x\in R$ is a lsd and $P$ is a prime ideal of $R$, then $x/1 \in R_P$ is a strong divisor. For the converse, use that if $x\in R$ is regular in every ring $R_P$, where $P$ is a prime ideal, so is $x$ because $R\to \prod [R_M \mid M \in \mathrm{Max}(R)]$ is injective. 
 The set of all locally strong divisors is a saturated mcs $\Lambda \Delta $.  Clearly, a strong divisor of a local ring is a lsd. Now if $R\subseteq S$ is a ring extension, we denote by $\lambda \delta (R)$ the ring $R_{\Lambda \Delta \cap \mathrm U(S)}$.  
 
   \begin{remark}\label{STASTRONG} Let $f: R\to S$ be a ring morphism.
   
(1)
If $f$ is a flat morphism and $x\in R$ is such that $f(x)$ is a lsd, then so is $x$. Indeed for $Q \in \mathrm{Spec}(S)$ lying over $P$, then $R_P \to S_Q$ is faithfully flat.
   
(2) If $f$ is a flat epimorphism and $x\in R$ is a lsd so is $f(x)$, because for each $Q\in \mathrm{Spec}(S)$ and $P:= f^{-1}(Q)$, the natural map $R_P \to S_Q$ is an isomorphism.
\end{remark}

\begin{theorem}\label{1.12.1} Let $R\subset S$ be an extension. Then  $R\subset S$ is a QR-extension if and only if each $T\in [R,S]$ of finite type over $R$ is of the form $T=R_s$ for some lsd $s\in R$. In particular, if these conditions hold, each $T\in [R,S]$ is of the form $T = R_{\mathcal T}$, where $\mathcal T \subseteq \Lambda \Delta$ is a mcs.
  \end{theorem}
\begin{proof} One implication is clear. Suppose that the extension is QR. To see that the condition holds, it is enough to suppose that it is of finite type. According to Lemma~\ref{1.9.1}, there is some $s\in R$ such that $S= R_s$. Then $R\subseteq S$ is Pr\"ufer, whence so is $R_P \subseteq S_P$ for each prime ideal $P$ of $R$ and $S_P =(R_P)_{s/1}$. We have also $S_P =(R_P)_y$, where $y \in \Delta(R_P)$ by Proposition~\ref{LOCALPRUF}. 
 It follows that $\mathrm{D}(s/1)=\mathrm{D}(y)$ and by Definition \ref{1.11}, 
 $s/1$ is a strong divisor. The last statement follows from the $(\mathcal{MCS}$)-rule applied to the flat epimorphism $R\subseteq T$, since $T$ is is a union of finitely generated QR-extensions.
\end{proof} 
  
\begin{theorem} Any extension $R\subset S$ has a QR-hull; that is, there exists a largest QR-extension $\chi(R)\in[R,S]$, contained in $\widetilde R$. As a consequence, $\chi(R)$ is the compositum  of all QR-extensions in $[R,S]$.
   \end{theorem}
\begin{proof} Let $X$ be the set of all QR-extensions in $[R,S]_{fg}$, which is directed upwards: take $T,U\in X$. They are of the form $R_x$ and $R_y$, where $x$ and $y$ are regular in $R$, because they are units in $S$. Then we have $R_x,R_y\subseteq R_{xy}$. We can now use the proof of \cite[Theorem 5]{DA} which holds for an arbitrary extension $R\subset S$ and show that $R_{xy} \in X$.
     
Denote by $\chi(R)$ the set union of the elements of $X$. Since a QR-extension in $[R,S]$ is a union of finitely generated QR-extensions, it is contained in $\chi(R)$. To complete the proof, observe that an element of $[R,\chi (R)]_{fg}$ is contained in an element of $X$, whence is in $X$, from which we infer that $R\subseteq \chi (R)$ is a QR-extension. 
         \end{proof}

Actually, the proof of Davis shows that the set of all elements $x\in R$ such that $R\to R_x$ is a QR-extension is a mcs $\Omega(R)$ (also denoted $\Omega $) contained in the mcs $\Lambda \Delta$.  Moreover, in case $R$ is either local or a Nagata ring, projective $R$-modules of rank one are free, so that an extension $R\subset S$ is Pr\"ufer if and only if it is a QR-extension by Proposition \ref{1.14}, giving  $\Omega\cap \mathrm U(S) = \Lambda \Delta \cap \mathrm U(S)$. 
An application of the $(\mathcal X)$-rule gives the following result.

\begin{corollary}\label{PRUFQR} If $R\subset S$ is an extension, then $\chi(R) =R_{\Omega\cap\mathrm U(S)}$. It follows that an extension $R\subset S$ is a QR-extension if and only if for each $s\in S$ there is some $\rho \in \Omega \cap \mathrm U(S)$ such that $\rho s \in R$. 
 \end{corollary}
 
We remark that $\beta(R)\subseteq\chi(R)\subseteq\lambda\delta(R)\subseteq \widetilde R$.

\begin{proposition} \label{PRUFQR1}An extension $R\subset S$ is a QR-extension if and only if each $S$-regular finitely generated ideal is equivalent to a principal ideal of $R$ and there exists a mcs $\Sigma \subseteq \Lambda \Delta\cap \mathrm U(S)$ such that $S =R_{\Sigma}$. 
    If these conditions hold, then $S=R_{\Lambda\Delta\cap \mathrm U(S)}$.
   \end{proposition}
 
 \begin{proof} 
Assume first that $R\subseteq S$ is a QR-extension. By the recall before Proposition \ref{1.27}, each $S$-regular finitely generated ideal is equivalent to a principal ideal of $R$. According to Theorem \ref{1.12.1}, there exists a mcs $\Sigma$, whose elements are some lsd of $R$, and such that $S=R_{\Sigma}$.
 
Conversely, assume that each $S$-regular finitely generated ideal is equivalent to a principal ideal of $R$ and there exists a mcs $\Sigma$ whose elements are some lsd of $R$, and such that $S=R_{\Sigma}$. Let $M\in\mathrm{Max}(R)$. It follows that $S_M=(R_{\Sigma})_M=(R_M)_{\Sigma '}$, where $\Sigma '$ is a mcs whose elements are some lsd of $R_M$. Then, Proposition \ref{1.7} implies that $R_M\subseteq S_M$ is Pr\"ufer. Since this holds for any $M\in\mathrm{Max}(R)$, we get that $R\subset S$ is Pr\"ufer, and then a QR-extension by the  recall before Proposition \ref{1.27}.

If these conditions hold, set $\Sigma':=\Lambda\Delta\cap\mathrm U(S)\subseteq S$. Since $\Sigma'$ is a mcs whose elements are units of $S$, it follows that $R_{\Sigma'}\subseteq S$. But $\Sigma\subseteq\Sigma'$ implies $S=R_{\Sigma}\subseteq R_{\Sigma'}\subseteq S$, so that $S=R_{\Sigma'}$.
  \end{proof}   
    
We end this section by considering ring extensions $R\subset S$ that are flat  epimorphisms, such that the support $\mathrm{Supp}(S/R)$ of the $R$-module $(S/R)$ is finite.   
We recall that $R\subset S$ is a flat epimorphism $\Leftrightarrow$  for all $P\in\mathrm{Spec}(R)$, either $R_P= S_P$ is an isomorphism or $S=PS$, these two conditions being mutually exclusive \cite[Proposition 2.4, p.112]{L}.

It is known that the support $\mathrm{Supp}(S/R)$ of the $R$-module $S/R$ is the set of all $P\in \mathrm{Spec}(R)$, such that $PS=S$. Therefore, each element of the support is $S$-regular. Moreover, the support is closed because as any support, it is stable under specialization. Hence the support equals to $\mathrm V(J)$, where $J$ is the intersection of all elements $P_1,\ldots, P_n$  of the support. Now each $P_i$ is the radical of an $S$-regular finitely generated ideal, as an examination of the proof of \cite[Corollary 13]{APRUF} by Abbas and Ayache shows.  
     Moreover, assume that $R\subset S$ is a QR-extension. Using \cite[Proposition 4.15, p.116]{KZ}, we get that $P_i$ is of the form $\sqrt{Rx_i}$ for some $x_i \in R$.
    Then $J=\sqrt{Rx}$ where $x=x_1\cdots x_n$. Now if $I$ is an $S$-regular finitely generated ideal and $Q$ is a prime ideal of $R$ containing $I$, then $Q$ is $S$-regular. Reasoning as above we see that $\sqrt I = \sqrt{Ry}$, for some $y\in R$. Taking into account the characterization of QR-extensions at the beginning of the section, we see that we have proved the following result:
    
\begin{proposition}\label{AYACHE} Let $R\subset S$ be a Pr\"ufer extension where $\mathrm{Supp}(S/R)$ is finite (in particular, if $R\subset S$ has FCP).
     We set $J:=\cap[P\mid P\in\mathrm{Supp}(S/R)]$. The following statements 	are equivalent:
    \begin{enumerate}
    \item  $R\subset S$ is a QR-extension;
\item  Each element of $\mathrm{Supp}(S/R)$ is equivalent to a principal ideal;
\item Each $S$-regular finitely generated ideal of $R$ is equivalent to a principal ideal.
      \end{enumerate}
In case one of the above statement holds, $J$ is a $S$-regular ideal  equivalent to a principal ideal $Rx$ and $\Gamma (\mathrm D(J)) = R_x$.
    \end{proposition}
\begin{proof} We only need to prove the following. By the flatness of the extension, $JS = \cap [PS \mid P \in \mathrm{Supp}(S/R)] =S$
 \end{proof}
 
\begin{remark} Proposition 20 of \cite{APRUF} states that if, in addition to the above hypotheses, $S$ is an integral domain, then each $T\in [R,S]$ is of the form $R_x$ for some $x\in R$. This proves that the extension is strongly affine.
 Actually, in the proof of \cite[Proposition 20]{APRUF}, we can replace the Kaplansky transform of an ideal  by a ring of sections.
 \end{remark}
 
  \section{Nagata extensions}
 
We start this section by recalling some facts about Nagata rings, that are explained in \cite[Section 3]{DPP3}. Let $R$ be a ring and $R[X]$ the polynomial ring in the indeterminate $X$ over $R$. We denote by $C(p)$ the content of any polynomial $p\in R[X]$. Then $\Sigma_R:=\{p\in R[X]\mid C(p)= R\}$ is a saturated mcs of $R[X]$, each of whose elements is a non-zero-divisor of $R[X]$. The {\it Nagata ring of} $R$ is defined to be $R(X):=R[X]_ {\Sigma_R}$. Its main properties that are used in this section are the following.  If $I$ is an ideal of $R$, we set $I(X)=IR(X)$, which is an ideal of $R(X)$. Now if $P$ is a prime ideal of $R$, then, $P(X)$ is a prime ideal of $R(X)$ lying over $P$. The inclusion map $R\hookrightarrow R(X)$ is a faithfully flat ring homomorphism, and Max$(R(X))=\{M(X)\mid M\in$ Max$(R)\}$. Hence $R(X)$ is local if $R$ is local. In fact, $R[X]\setminus \Sigma_R=\cup\{M[X]\mid M\in\mathrm{Max}(R)\}$. Also, note that any ring homomorphism $f:R\to S$ extends to a ring homomorphism $f_e:R[X]\to S[X]$ that fixes $X$, and $f_e$ in turn induces a ring homomorphism $f_{nag}:R(X)\to S(X)$ that also fixes $X$. By the  remark before \cite[Proposition II.9]{DMPP}, if $f$ is an extension, then so is $f_{nag}$. Finally, if $f: R\to S$ is a ring homomorphism, we will say that $R(X)\otimes_R S=S(X)$ {\it canonically} ({\it with respect to} $f$) if the ring homomorphism $g: R(X)\otimes_RS \to S(X)$ that is induced by $f$ is an isomorphism. We will also say that  $f$ verifies the  property $(\mathcal T)$. 
  An FCP extension $f$ verifies $(\mathcal T)$ \cite[Proposition 3.2]{DPP3}.
 
One sees easily that $R(X)_{MR(X)}=R_M(X)$ for each 
  $M\in\mathrm{Spec}(R)$.
  
 We recall some   results frequently used in the section. 

\begin{proposition}\label{TENS} \cite[Section IV, Proposition 4 and Proposition 7]{PicC}
A ring morphism $f:R\to S$  verifies $(\mathcal T)$ when it is either integral (no necessarily injective) or a flat epimorphism. Consequently, for each $P\in \mathrm{Spec}(R)$, the ring morphism $R_P\to S_P$ verifies $(\mathcal T)$ (we say that the morphism verifies $(\mathcal T_P)$).
\end{proposition}

We deduce from the above proposition:

(A) If $\Sigma$ is a mcs of $R$ then $R_{\Sigma} (X)= (R(X))_{\Sigma}$.

(B) If $I$ is an ideal of $R$, then $(R/I)(X) = R(X)/I(X)$.

\begin{proposition}\label{LOCALI}
If $f: R\to S$ verifies $(\mathcal T)$ and $(\mathcal T_M)$ for some maximal ideal $M$ of $R$, then $R(X)_{M(X)} \to S(X)_{M(X)}$ identifies to $R_M(X) \to S_M(X)$.
\end{proposition}
\begin{proof} We have $S_M(X){{\mathcal T}_M\atop =}S_M\otimes_{R_M}R_M(X)=(S\otimes_RR_M)\otimes_{R_M}R_M(X)=S\otimes_R R_M(X)=S\otimes_RR(X)_{M(X)}= S\otimes_R R(X)\otimes_{R(X)}R(X)_{M(X)}{{\mathcal T} \atop =} S(X)_{M(X)}$.
\end{proof}

\begin{lemma}\label{LYINGOVER} Let $R\subseteq S$ be an extension and $Q$ a prime ideal of $S$, lying over $P$ in $R$. Then $Q(X)$ is lying over $P(X)$ and the natural map $R(X)_{P(X)} \to S(X)_{Q(X)}$ identifies to $R_P(X) \to S_Q(X)$.
\end{lemma}
\begin{proof} That $Q(X)$ is lying over $P(X)$ is an easy consequence of the following fact: if $p \in R[X]$ and $q\in R[X]$ is such that  $C(q)=R$, then $C(pq)= C(p)$. This  follows from the Lemma of Dedekind-Mertens.
\end{proof}

\begin{proposition}\label{MONTE} \cite[Corollary 3.15]{Pic 5} Let $R\subseteq S$ be an extension of rings, $R\to R'$ a faithfully flat base change and $S':= R'\otimes_RS$.
 If $R'\to S'$ is Pr\"ufer, then $R\to S$ is Pr\"ufer.
 \end{proposition}

\begin{proposition}\label{EPINAG}  Let $R \to S$  be  a ring extension. The following statements are equivalent:

\begin{enumerate}

\item  $R \to S$  is a flat epimorphism;

\item $R(X) \to S(X)$ is a flat epimorphism; 

\item $R[X]_\Sigma =S(X)$, where $\Sigma \subseteq R[X]$ is the mcs     of all $ p(X) \in R[X]$ whose contents in  $S$ are $S$.

\end{enumerate}
\end{proposition}
\begin{proof} (1) is equivalent to (3) by \cite[Lemma 3.11]{PT}. Now $R(X) \to S(X)$ is a flat epimorphism if $R\to S$ is a flat epimorphism \cite[Lemma 3.1(g)]{DPP3} and (Proposition~\ref{TENS}). 
 To prove the converse we use \cite[Scholium A (1)]{Pic 5} which states that a ring morphism $A\to B$ is a flat epimorphism if and only if its local morphisms
are isomorphisms and $A \to B$ is an $i$-morphism. We then consider a prime ideal $Q$ of $S$ lying over $P$. Suppose that $R(X)\to S(X)$ is a flat epimorphism. From  Lemma \ref{LYINGOVER}, we get that $R_P(X) \cong R(X)_{P(X)}\to S_Q(X)\cong S(X)_{Q(X)}$, and deduce that $R_P(X)\to S_Q(X)$ is an isomorphism, whence $R_P\to S_Q$ is injective. Identifying this map with an extension $R_P\subset S_Q$, we get that $R_P(X) =S_Q(X)$. It follows from \cite[Lemma 3.1(f)]{DPP3} that $R_P=S_Q\cap R_P(X)=S_Q\cap S_Q(X)=S_Q$. The injectivity of $\mathrm{Spec}(S)\to\mathrm{Spec}(R)$ may be proved as follows. If $Q$ and $Q'$ are prime ideals of $S$ lying over $P$ in $R$, then $Q(X)$ and $Q'(X)$ are lying over $P(X)$.
\end{proof}

\begin{theorem}\label{A0.7.3} An extension $R\subseteq S$ is (quasi-)Pr\"ufer if and only if $R(X) \subseteq S(X)$ is (quasi-)Pr\"ufer.
 In that case, $\overline{R(X)}=\overline R(X)\cong \overline R \otimes_R R(X)$ and $S(X)\cong S\otimes_R R(X)$.
   \end{theorem}
\begin{proof} Assume that $R\subset S$ is Pr\"ufer, we intend to show that so is $R(X)\subset S(X)$. We can suppose that $(R,M)$ is local according to Proposition \ref{LOCALI}, in order to use Proposition~\ref{LOC}  and Proposition~\ref{VALU}. Then it is  enough to know the following facts: $V(X)$ is a valuation domain if so is $V$ and $R[X]_{P[X]} \cong R(X)_{P(X)} \cong R_P(X)$ \cite[Theorem 4 and Lemma 2]{AR}; $ R(X)/P(X)\cong (R/P)(X) $ for $P\in \mathrm{Spec}(R)$. Now suppose that $R(X)\subset S(X)$ is Pr\"ufer, then $R\subset S$ is a flat epimorphism by Proposition~\ref{EPINAG}. Hence,  $S(X)\cong R(X)\otimes_RS$ by Proposition \ref{TENS}. It follows that Proposition~\ref{MONTE} entails that $R\subset S$ is Pr\"ufer since $R \to R(X)$ is faithfully flat. For the quasi-Pr\"ufer case, it is enough to use the Pr\"ufer case and Proposition~\ref{TENS}, because $\overline{R(X)} = \overline{R}(X)$. 
   \end{proof}
     
\begin{corollary} An element $x$ of ring $R$ is a lsd of $R$ if and only if $x$ is a  lsd of  the  ring $R(X)$.
   \end{corollary}
  
\begin{proof} Obvious by Theorem \ref{A0.7.3} and Proposition \ref{1.9.1.0}.
      \end{proof}
   
\begin{proposition} If $R\subset S$ is an extension, then $R(X)\subset S(X)$ is a Pr\"ufer extension if and only if it is a QR-extension, and in this case $S(X) = R(X)_{\Lambda \Delta (R(X)) \cap \mathrm U(S(X))}$.
   \end{proposition}
   
\begin{proof} If $R(X)\subset S(X)$ is Pr\"ufer, it is a QR-extension because over $R(X)$ each projective module of rank one is free (Propositions \ref {1.14},  \ref{PRUFQR1} and \cite[Theorem 3.1]{FT}).
    \end{proof}
\begin{proposition}\label{1.10.1.1} Let $R\subset S$ be a Pr\"ufer extension of finite type. 
    Then $S(X)= R(X)_{g(x)} $ for some lsd $g(X) \in R(X)$.
 \end{proposition}
\begin{proof} We know that $R(X) \subset S(X)$ is Pr\"ufer. To complete the proof, use  the above result and Theorem~\ref{1.12.1}.  
  \end{proof}
  
If $R$ is a local ring, within a unit, a non-unit strong divisor of the local ring  $R(X)$ can be written $p(X)/1$ where $p(X)\in R[X]$ is regular with content $c(p(X)) \neq R$.  Then $R(X)p(X)/1$ is comparable to each  ideal. 
 We note that if $x$ is an element of $R$, which is a strong divisor in $R(X) $, it is a strong divisor of  $R$ by Proposition \ref{1.9.2}.
   
\begin{theorem}\label{NAGstrong} Let $R$ be a local ring, $p(X)/1$ a non unit strong divisor of $R(X)$ and $I$ the content ideal of $p(X)=\sum_{i=0}^na_iX^i \in R[X]$. Then $I=R\rho$ where $\rho$ is a strong divisor of $R$ such that $p(X)/1\simeq \rho$ in $R(X)$.
  In fact, $I(X)=R(X)p(X)/1$. If $p(X)/1\in R(X)a_i$ for some $a_i$, then $I=Ra_i$.
Moreover $ \mathcal X_R (R(X)_{p(X)/1}) =\mathrm D(I)$.
  \end{theorem}

  \begin{proof}
We start by showing that $I(X)=R(X)p(X)/1$. We have that $I=\sum_{i=0}^nR a_i$ and $p(X)/1\in I(X)$. Moreover, $p(X)/1$ being a strong divisor of $R(X)$, the ideal $R(X)p(X)/1$ is comparable to each $R(X)a_i$. Assume first that $a_i\in p(X)R(X)$ for each $a_i$. Then, $I(X)\subseteq R(X)p(X)/1\subseteq I(X)$ gives $I(X)=R(X)p(X)/1$. Assume now that $p(X)/1\in R(X)a_i$ for some $a_i$. Then, $p(X)/1=a_iq(X)/s(X)$ for some $q(X),s(X)\in R[X]$ with $c(s)=R$ and $ a_i$ is regular. This implies that $p(X)s(X)=a_iq(X)$ in $R[X]$. By the Lemma of Dedekind-Mertens, $c(ps)=c(p)=I=a_ic(q)\subseteq Ra_i\subseteq I$, gives $I=Ra_i$. Moreover, $a_ic(q)=c(p)=Ra_i$ leads to $c(q)=R$. To conclude, $ a_i/1=p(X)s(X)/q(X)$ and we recover $I(X)=R(X)p(X)/1$. In both cases, $I(X)=R(X)p(X)/1$ shows that $I(X)$ is  a strong divisor of $R(X)$. 
  
We claim that $I$ is a strong divisor of $R$. Let $x\in R$. Then, $xR(X)$ is comparable to $I(X)$. If $x\in I(X)$, then, $x/1=r(X)/t(X)$, with $r(X)\in I[X]$ and $t(X)\in R[X]$ with $c(t)=R$. It follows that $xt(X)=r(X)$ in $R[X]$, so that $c(xt)=Rx=c(r)\subseteq I$. If $I(X)\subseteq xR(X)$, let $y\in I$. We have $y/1=xu(X)/v(X)$, with $u(X),v(X)\in R[X]$ and $c(v)=R$, so that $yv(X)=xu(X)\in R[X]$, giving $yc(v)=Ry=xc(u)\subseteq Rx$. Then, $I\subseteq Rx$. In any case, $Rx$ is comparable to $I$, and $I$ is  a strong divisor of $R$.
  
Setting $S:=R(X)_{p(X)/1}$, we have that $R(X)\subseteq S$ is Pr\"ufer.  Because $p(X)\in IR[X]$, $I(X)$ is $S$-regular and then we have a tower of extensions $R(X)\subseteq\Gamma(\mathrm D(I(X)))\subseteq S$ by Remark~ \ref{RS}(2). It follows that $R(X)\subseteq\Gamma(\mathrm D(I(X)))$ is Pr\"ufer. Because $R\to R(X)$ is faithfully flat, $\Gamma(\mathrm D(I(X))) \cong\Gamma(\mathrm D(I))\otimes_R R(X)$ by Proposition~\ref{DOM}(3).  According to Proposition~\ref{FAITHPRUF}, faithfully flat morphisms descend Pr\"ufer extensions and then $R\to\Gamma(\mathrm D(I))$ is Pr\"ufer. Since $I(X)=R(X)p(X)/1$, we get $IS=S$, then $\Gamma(\mathrm D(IS))=\Gamma (\mathrm D(S)) =S$. 
 
It follows that $R\to\Gamma(\mathrm D(I))$ is of finite presentation because this property is descended by faithfully flat ring morphism \cite[Proposition 5.3]{OLPUR}. Therefore, there is a strong divisor $t\in R$, such that $\Gamma (\mathrm D(I))  = R_t$ by Proposition~\ref{LOCALPRUF}.
  Since $I(X)$ is a principal ideal generated by a regular element, it is free, and then a flat ideal. But $I(X)\cong I\otimes_RR(X)$ by flatness of $R\to R(X)$. Because faithfully flat morphisms descend flatness \cite[Proposition 4.1]{OLPUR}, $I$ is a flat ideal of finite type, whence is free over $R$. Therefore, $I$ is a principal ideal $R\rho$. 
  Then, $I(X)=\rho R(X)=R(X)p(X)/1$ shows that $\rho$ is a strong divisor of $R$.
  
The image of $\mathrm{Spec}(S)\to\mathrm{Spec}(R)$ is $\mathrm D(I)$,  because ${}^ag(\mathrm D(p(X)))= \mathrm D(I)$.
We get $\mathcal X_R(S)\subseteq\mathrm D(I)$ by Remark \ref{RS}(1). Let $P\in\mathrm D(I)$, so that there exists some coefficient of $p(X)$ which is not in $P$, giving $p(X)/1\not\in P(X)$, which is a prime ideal of $R(X)$. Then, there exists $Q\in\mathrm{Spec}(S)$ lying over $P(X)$, and then over $P\in\mathrm{Spec}(R)$. It follows that $P\in\mathcal X_R(S)$, giving $\mathcal X_R(S)= \mathrm D(I)$.
   \end{proof} 
  
  \begin{remark}\label{NAGMIN} If $R \subset S$ is a Pr\"ufer extension of finite type (for example an FCP extension or a minimal extension) and $R$ is a local ring, there is a strong divisor $\rho$ of $R$, such that  $S= R_\rho$.  It follows that $S(X)= R(X)_\rho$.
  As $R(X)\subset S(X)$ is of finite type and Pr\"ufer, $S(X)$ is of the form  $R(X)_{p(X)/1}$, where $p(X)/1$ is a strong divisor of $R(X)$. We therefore have $p(X)/1\simeq\rho$. We recover (Proposition~\ref{1.10.1.1}) under a more precise form.
  
 Here is an example illustrating Proposition~\ref{NAGstrong}. 
  
Let $R$ be a local arithmetical ring. According to Example~\ref{ARITH}, each of its regular elements is a strong divisor. Now let $p(X)/1$ be a regular element of $R(X)$, where $p(X)\in R[X]$ has a content $\neq R$. If $a_0,\ldots, a_n$ are the coefficients of $p(X)$, one of them, say $a_i$, is a multiple of all the others. It follows that $p(X)=a_iq(X)$ where the content of $q(X)$ is $R$ and $a_i$ is regular. Therefore, $a_i$ is a strong divisor and $p(x)/1\simeq a_i$. Note that $R(X)$ is local and arithmetical \cite[Theorem 3.1]{AM}.
  \end{remark}
  
We now give some conditions implying that the Pr\"ufer hull commutes with the formation of Nagata rings. We don't know if these conditions are superfluous.   
  
We say that a ring $R$ is quasi-Pr\"uferian if $R\to R(X)$ is an $i$-morphism. These rings will be studied in Section 10. 
    
\begin{proposition} Let $R\subseteq S$ be a ring extension, where $R$ is quasi-Pr\"uferian and $\widetilde R\subseteq S$ is lying over (for example if the extension is almost-Pr\"ufer), then $\widetilde R(X) = \widetilde{R(X)}$.
\end{proposition} 
\begin{proof} If $R$ is quasi-Pr\"uferian so is $\widetilde R$, because $R \subseteq \widetilde R $ is a flat epimorphism, according to Proposition~ \ref{INTQP}(1). Therefore $\mathrm{Spec}(\widetilde R(X))\to\mathrm{Spec}(\widetilde R)$ is bijective. The extension $\widetilde R(X)\to\widetilde{R(X)}$ is Pr\"ufer, whence a flat epimorphism, which is surjective if it has lying-over. A prime ideal of $\widetilde R(X)$ is of the form $P(X)$, where $P$ is a prime ideal of $\widetilde R$. Let $Q$ be a prime ideal of $S$ above $P$. Then $Q(X)$ contracts to $P(X)$, which gives a prime ideal of $\widetilde {R(X)}$ lying over $P(X)$. 
\end{proof}

 \begin{proposition}\label{A0.7.7} Let $R\subseteq S$ be a ring extension such that the map $\psi:[R,S]\to[R(X),S(X)]$ defined by $T\mapsto T(X)$ is bijective. Then,  $\widetilde{R(X)} = \widetilde{R}(X)$.
\end{proposition}
\begin{proof} According to Theorem \ref{A0.7.3}, $R(X)\subseteq \widetilde{R}(X)$ is Pr\"ufer. Then, $\widetilde{R}(X)\subseteq \widetilde{R(X)}$. Assume that $\widetilde{R(X)}\neq\widetilde{R}(X)$; so that, there exists $T\in[R,S],\ T\neq\widetilde{R}$ such that $T(X)=\widetilde{R(X)}$ because $\psi$ is bijective. Using again Theorem \ref{A0.7.3}, it asserts that $R\subseteq T$ is Pr\"ufer, leading to $T\subseteq\widetilde{R}$, and then to $T(X)=\widetilde{R(X)}\subseteq \widetilde{R}(X)\subseteq\widetilde{R(X)}$. To conclude, $\widetilde{R(X)} = \widetilde{R}(X)$.
\end{proof}

\begin{proposition}\label{A0.7.8} Let $R\subseteq S$ be a ring extension, where $R$ is local quasi-Pr\"uferian. Then, $\widetilde{R(X)}= \widetilde{R}(X)$.
\end{proposition}
\begin{proof} Since $R(X)\subseteq\widetilde{R(X)}$ is Pr\"ufer and $R(X)$ is local, Proposition~\ref{VALU} says that there exists $Q'\in \mathrm{Spec}(R(X))$ such that $\widetilde{R(X)}=(R(X))_{Q'}$. In the same way, there exists $P'\in\mathrm{Spec}(R(X))$ such that $\widetilde{R}(X)=(R(X))_{P'}$. We still have $\widetilde{R}(X)\subseteq\widetilde{R(X)}$, so that $Q'\subseteq P'$.  
Set $P:=P'\cap R$ and $Q:=Q'\cap R$. As $R$ is quasi-Pr\"uferian, it follows that $P'=P(X)$ and $Q'=QR(X)$, with $Q\subseteq P$. 
Then,  $\widetilde{R}(X)=(R(X))_{P'}=R_P(X)$. A same reasoning gives $\widetilde{R(X)}=(R(X))_{Q'}=R_Q(X)\subseteq S(X)$. 
In particular,  $\widetilde{R}= \widetilde{R}(X)\cap S=R_P(X)\cap S=R_P$. Now, $R_Q\subseteq R_Q(X)\subseteq S(X)$ shows that $R_Q\in[R,S]$. But, $R\subseteq R_Q$ is Pr\"ufer, since so is $R(X)\subseteq R_Q(X)=\widetilde{R(X)}$. Then, we have $R_Q\subseteq R_P\subseteq R_Q$ which implies $R_Q= R_P$, and then $R_Q(X)=R_P(X)$, that is $\widetilde{R(X)} = \widetilde{R}(X)$.
\end{proof}

\begin{proposition}\label{A0.7.4} If $R \subseteq S$ is almost-Pr\"ufer, then  $R(X)\subseteq S(X)$ is almost-Pr\"ufer, so that $\widetilde{R(X)}= \widetilde{R}(X)$.
\end{proposition}
\begin{proof} 
If $R\subseteq S$ is almost-Pr\"ufer, then $R\subseteq\widetilde{R}$ is Pr\"ufer and $\widetilde{R}\subseteq S$ is integral; so that, $R(X)\subseteq \widetilde{R}(X)$ is Pr\"ufer and $\widetilde{R}(X)\subseteq S(X)$ is integral, whence $R(X)\subseteq S(X)$ is almost-Pr\"ufer with $\widetilde{R(X)} = \widetilde{R}(X)$.
\end{proof}

In case $R\subseteq S$ is an FCP extension, we get more results. A minimal extension $R\subset S$ is such that there exists a maximal ideal $M$ of $R$ satisfying $\mathrm{Supp}(S/R)=\{M\}$. Such a prime ideal $M$ is called the {\it crucial (maximal) ideal} $\mathcal{C}(R,S)$ of $R\subset S$ \cite[Theorem 2.1]{DPP2}. 

\begin{lemma}\label{A0.7.5} Let $R\subset S$ be an FCP ring extension such that $\widetilde{R}=R$. Then, $\widetilde{R (X)}=R(X)$.
\end{lemma}

\begin{proof} If $R(X)\neq\widetilde{R(X)}$, there is some $T'\in[R(X),\widetilde{R (X)}]$ such that $R(X)\subset T'$ is Pr\"ufer minimal. Set $M':= \mathscr{C}(R(X),T')$ which is in $\mathrm{MSupp}(S(X)/R(X))$. Let $M\in\mathrm{MSupp}(S/R)$ be such that $M'=MR(X)$ \cite[Lemma 3.3]{DPP3}. Since \cite[ Proposition 4.18(2)]{Pic 5} asserts that $M'\not\in\mathrm{MSupp}(\overline{R(X)}/R(X))=\mathrm{MSupp}(\overline{R}(X)/R(X))$, this gives that $M\not\in\mathrm {MSupp}(\overline{R}/R)$, which entails that $M\in\mathrm{MSupp}(S/\overline{R})$. By \cite[Lemma 1.8]{Pic 6}, there exists $T\in[R,\widetilde{R}]$ such that $M=\mathcal C (R,T)$ with $R\subset T$ Pr\"ufer minimal since $T\not\in[R,\overline{R}]$, a contradiction.
\end{proof}

\begin{proposition}\label{A0.7.6} If $R\subset S$ has FCP, then, $\widetilde{R}(X)=\widetilde{R (X)}$.
\end{proposition}

\begin{proof}  
Because $R\subseteq\widetilde{R}$ is Pr\"ufer, $R(X)\subseteq\widetilde{R}(X)$ is Pr\"ufer by Proposition~\ref{A0.7.3}. Then, $\widetilde{R}(X)\subseteq\widetilde{R (X)}$. Assume that $\widetilde{R}(X)\neq\widetilde{R(X)}$ and set $T:=\widetilde{R}$, so that $T=\widetilde{T}$, giving $\widetilde{T(X)}=T(X)=\widetilde{R}(X)$ by Lemma~\ref{A0.7.5}. Hence $\widetilde{T(X)}\subset\widetilde{R(X)}$ is a Pr\"ufer extension,  contradicting  the definition of $\widetilde{T (X)}$. So, $\widetilde{R}(X)=\widetilde{R (X)}$.
\end{proof}
 
 \section{Quasi-Pr\"uferian rings}

We call a ring $R$ {\it quasi-Pr\"ufer} if $R\subseteq\mathrm{Tot}(R)$ is quasi-Pr\"ufer in order to have a coherent definition ($R$ is classically called Pr\"ufer if $R\subseteq\mathrm{Tot}(R)$ is Pr\"ufer). Trivially, a total quotient ring is quasi-Pr\"ufer.  
Proposition \ref {NAGTOT} gives another   example.

We recall that a ring $R$ is called {\it McCoy} (or satisfies the condition (A)) if each finitely generated ideal $I$ of $R$ contained in $\mathrm{Z}(R)$ is such that $0:I\neq 0$ 
\cite{LUPROPA}.

 It is easy to prove that if $R\to S$ is an injective flat epimorphism and $R$ is McCoy, then $S$ is McCoy. Indeed, a finitely generated ideal of $S$ is of the form $IS$ for some finitely generated ideal $I$ of $R$.
It is well known that $R[X]$ is McCoy for any ring $R$, whence so is $R(X)$ 
\cite[Corollary 1]{HUC}.

We will need the   following definitions and results:

Let $M$ be an $R$-module and $P\in\mathrm{Spec}(R)$. Then $P$ is called  an {\it attached prime ideal} of $M$ (a strong Krull prime ideal by \cite{IR}) if for each finitely generated ideal $I \subseteq P$ there is some $x\in M$, such that $I\subseteq 0:x \subseteq P$. 
 The set of all attached prime ideals of $M$ is denoted by  $\mathrm{Att}(M)$. 
 We recall that $\mathrm{Z}(R) = \cup [ P \mid P \in \mathrm{Att} (R) ] $.

(1) An element of $\mathrm{Att}(R[X])$ is of the form $P[X]$ for some $P \in \mathrm{Att}(R)$ \cite[Theorem 2.5]{IR}. Therefore, the set of zero-divisors of $R[X]$ is $ \mathrm{Z}(R[X]) = \cup [P[X] \mid  P \in \mathrm{Att}(R)]$. 

(2) The mcs $S:=\{p(x)\in R[X]\mid\mathrm{c}(p(X))=R\}$ is contained in $R[X]\setminus\mathrm{Z}(R[X])$ and the elements of $\mathrm{Att}(R[X])$ are of the form $P(X)$, where $P\in\mathrm{Att}(R)$. It follows that $R[X]\subseteq R(X)\subseteq\mathrm{Tot}(R[X])$ and then $\mathrm{Tot}(R(X))= \mathrm{Tot}(R[X])$.

Lucas proved that a ring $R$ is McCoy if and only if $ \mathrm{Tot}(R)(X)=\mathrm{Tot}(R[X])(= \mathrm{Tot}(R(X)))$ by  \cite[Proposition 4.1]{LUPROPA} and the above result (2).

\begin{proposition}\label{NAGTOT} Let $R$ be  a quasi-Pr\"ufer McCoy ring, then $\mathrm{Tot}(R(X))$

\noindent$\cong\mathrm{Tot}(R)(X))$ and $R(X)$ is quasi-Pr\"ufer.
\end{proposition}
\begin{proof}
We have $\mathrm{Tot}(R(X))=(\mathrm{Tot}(R))(X)$, so that $R(X) \subseteq  \mathrm{Tot}(R(X))$ is quasi-Pr\"ufer, because deduced from the flat epimorphism $R\to \mathrm{Tot}(R)$ by the base change $R\to R(X)$.
\end{proof}

Quasi-Pr\"ufer rings defined in the book \cite{FHP} do not coincide with our quasi-Pr\"ufer rings. They are called elsewhere quasi-Pr\"uferian, a terminology we adopt.  

\begin{definition} \cite{ACE} A ring $R$ is called {\it quasi-Pr\"uferian } if for each prime (resp.; maximal) ideal $M$ of $R$, any prime ideal $Q$ of $R[X]$, such that $Q\subseteq M[X]$ is of the form $(Q\cap R)[X]$. It is easy to show that a ring $R$ is quasi-Pr\"uferian if and only if the faithfully flat ring morphism $R\to R(X)$ is an $i$-extension (because any maximal ideal of $R(X)$ is of the form $MR(X)$ for some $M\in\mathrm{Max}(R)$). Another characterization is $\mathrm{Spec}(R(X))=\{P(X)\mid P\in\mathrm{Spec}(R)\}$. We will use the following characterization: a ring $R$ is quasi-Pr\"uferian if and only if $R\to R(X)$ has the incomparability property (the INC-property).
\end{definition}

The class of quasi-Pr\"uferian rings is stable under the formation of factor rings  and the formation of Nagata rings. Moreover, a ring $R$ is quasi-Pr\"uferian if and only if $R_P$ is quasi-Pr\"uferian for each $P\in\mathrm{Spec}(R)$. Over an integral domain the two classes, quasi-Pr\"ufer and quasi-Pr\"uferian,  coincide \cite[Section 6.5]{FHP}. It follows that a ring which is locally a quasi-Pr\"ufer domain is  quasi-Pr\"uferian.

The stability of the class of quasi-Pr\"uferian rings under various operations does not seem to be valid for the class of quasi-Pr\"ufer rings. For example, the formation of total quotient rings does not commute with localizations. We will show at the end of the Section (Example \ref{UFD}) that these classes are different.

We will use and sometimes generalize  some results of \cite{ADF}  holding in the integral domain context.

We intend now to generalize \cite[Lemma 2.3 and Theorem 2.7]{ADF}.

\begin{proposition}\label{INTQP} Let $f: R \to S$ be a ring morphism.
\begin{enumerate}
\item If $R$ is quasi-Pr\"uferian and $f$ is either an integral morphism or a flat epimorphism, then $S$ is quasi-Pr\"uferian.

\item If $S$ is quasi-Pr\"uferian and $f$ is injective and integral, then $R$ is quasi-Pr\"uferian.

\item If $R$ is quasi-Pr\"uferian and $f$ is a quasi-Pr\"ufer extension, then $S$ is quasi-Pr\"uferian.
\end{enumerate}
\end{proposition}
\begin{proof}
(1) If $R$ is quasi-Pr\"uferian and $f$ is either integral or a flat epimorphism, then $S(X)\cong R(X)\otimes_RS$ by Proposition \ref{TENS}, so that $R(X) \to S(X)$ has the INC-property and so has $R\to R(X)$. It follows easily that $S \to S(X)$ has the INC-property.

(2) Suppose that $f$ is integral and injective, then so is $R(X) \to S(X)$. Therefore, each couple $P\subseteq P'$ of prime ideals of $R(X)$ can be lifted up to a couple $Q\subseteq Q'$ of prime ideals of $S(X)$, by the lying-over and going-up properties of an integral extension. Because
$S \to S(X)$ has the INC-property, as well as $R \to S$,  
we can assert that $R \to R(X)$ has the INC-property.

(3) We have a tower of extensions $R \subseteq \overline R \subseteq S$, where the last extension is Pr\"ufer.
\end{proof}

As usual if $P$ is a prime ideal of a ring $R$, a prime ideal $Q$ of $R[X]$ is called an {\it upper} of $P$ if $Q \neq P[X]$ and $P=Q\cap R$.

\begin{theorem}\label{GENQP} Each next statement  on a ring $R$ implies the following: 
\begin{enumerate}

\item $R$ is  quasi-Pr\"uferian,
\item If $P\in \mathrm{Spec}(R)$ and $M \in \mathrm{Max}(R)$, then no upper of $P$ is contained in $M[X]$,

\item If $M$ is a maximal ideal of $R$, then no upper of a minimal prime ideal of $R$ is contained in $M[X]$,

\item Any flat injective epimorphism  $R \subseteq S$ is quasi-Pr\"ufer,

\item $R\subseteq\mathrm{Tot}(R)$ is quasi-Pr\"ufer, {\it i.e.} $R$ is quasi-Pr\"ufer.
\end{enumerate}
\end{theorem}

\begin{proof}  We will follow the scheme of the proof of \cite[Theorem 2.7]{ADF}. 
We first prove the contrapositive of (4) $\Rightarrow$ (3). Suppose that an injective flat epimorphism $R\subseteq S$ is not quasi-Pr\"ufer. Then there exists some $u\in S$ such that $R\subset R[u]$ does not satisfies INC. Hence by \cite[Theorem 2.3]{DLO}, there exist distinct prime ideals $Q'_1\subset Q'_2$ of $R[u]$ and a maximal ideal $M$ of $R$, such that $Q'_1\cap R = Q'_2\cap R=M$. Consider the ring morphism $e:R[X]\to R[u]$, defined by $X \mapsto u$ and set $Q_i:= e^{-1}(Q'_i)$. Let also be a minimal prime ideal $N' \subset Q'_1$ of $R[u]$ and set $N=e^{-1}(N')$. Now because $R[u]\subset S$ is injective there is a minimal prime ideal  $W$ of $S$ lying over $N$. Because $R\subset S$ is flat, whence has the going-down property, we get that $W\cap R=N\cap R$ is a minimal prime ideal $\Pi$ of $R$.
As $Q_1\subset Q_2$ are distinct prime ideals ($e$ is a surjective map) of $R[X]$ lying over $M$, it follows that $Q_1=M[X]$ 
 (it is enough to consider the field $R/M$).

Suppose that $N$ is not an upper of $\Pi$, then $N=\Pi[X]$. We then have a factorization $R/\Pi\to(R/\Pi)[X]=R[X]/\Pi[X]=R[X]/N\cong R[u)/N'\to S/W$ into injective morphisms. An appeal to \cite[Corollaire 3.2(ii), p. 114]{L} yields that $R/\Pi\to(R/\Pi)[X]$ is a faithfully flat epimorphism, whence an isomorphism, an absurdity. Since $Q_1= M[X]\supseteq N$, then (3) fails.
\end{proof}

\begin{lemma}\label{TREED} A ring $R$, with $R(X)$ treed, is treed and quasi-Pr\"uferian.
\end{lemma}
\begin{proof}  \cite[Proposition 2.2]{ADF} is valid for arbitrary rings.
\end{proof}

\begin{proposition}\label{GENQPC}   Let $R$ be a ring such that $\mathrm{Tot}(R)$ is zero-dimensional. Then $R$ is quasi-Pr\"uferian if and only if $R$ is quasi-Pr\"ufer. In this case,  $R(X)$ is treed and $\mathrm{Tot}(R(X)) \cong \mathrm{Tot}(R)\otimes_RR(X) \cong (\mathrm{Tot}(R))(X)$.
\end{proposition}
\begin{proof} If $R$ is quasi-Pr\"uferian, then $R$ is quasi-Pr\"ufer by Theorem \ref{GENQP}. To prove the converse, and according to Lemma~\ref{INTQP}, we can suppose that $R$ is Pr\"ufer. Then we can reduce to the case where $R$ is local, because for any prime ideal $P$ of $R$, we have $(R(X))_P\cong R_P(X)$ and $\mathrm{Tot}(R)_P \cong  \mathrm{Tot}(R_P)$, so that $\mathrm{Tot}(R_P)$ is zero-dimensional. The proof of the last statement is as follows. There is an injective ring morphism $\mathrm{Tot}(R)\to\mathrm{Tot}(R_P)$ because $R_P$ is flat over $R$. This morphism induces another one $\mathrm{Tot}(R)\otimes_RR_P \to \mathrm{Tot}(R_P)$. There is a factorization $\mathrm{Tot}(R) \to\mathrm{Tot}(R)\otimes_R R_P\to\mathrm{Tot}(R_P)$. The first morphism is a flat epimorphism and the composite is injective.  
   We get that $\mathrm{Tot}(R)\otimes_R R_P\to\mathrm{Tot}(R_P)$ is injective by \cite[Lemme 3.4, p.114]{L}.
    To conclude, we use the following fact: an injective flat epimorphism, whose domain is zero-dimensional, is an isomorphism \cite[Lemme 1.2, p. 109]{L}, whence the injective flat epimorphism $\mathrm{Tot}(R)\otimes_R R_P\to \mathrm{Tot}(R_P)$ is an isomorphism.

We claim that $R$ is treed if $R\subset\mathrm{Tot}(R)$ is Pr\"ufer when $R$ is local. In this case, there exists some divided prime ideal $P$ of $R$ such that $\mathrm{Tot}(R)=R_P$ and $R/P$ is a valuation domain. As $T:=R_P$ is zero-dimensional, $P$ is a minimal prime ideal of $T$. In fact, $\mathrm{Spec}(T)=\{P\}$. Then $P$ is also a minimal prime ideal of $R$ because $T=R_P$. Moreover, $R\to T$ being injective, $P$ is the unique minimal prime ideal of $R$. At last, $\mathrm{Spec}(R)=\mathrm{V}(P)$ is a chain because $R/P$ is a valuation domain, so that $R$ is treed. 

The same proof works for $R(X)$,  which is also local and such that $R(X)\subset T(X)$ is Pr\"ufer. 
  To begin with, $T$ being zero-dimensional, so is $T(X)$, because the (minimal prime ideals) maximal ideals of $T(X)$ are of the form $MT(X)$ where ($M\in\mathrm{Min}(T)$) $M\in\mathrm{Max}(T)$. Then there is a factorization $R(X)\to\mathrm{Tot}(R)\otimes_R R(X)\to\mathrm{Tot}(R(X))$, where the first morphism is a flat epimorphism and the composite is an injective flat epimorphism. It follows that the last morphism, being a flat injective epimorphism, is an isomorphism because its domain is zero-dimensional. 
  Therefore, $T(X)\cong T\otimes_R R(X)\cong\mathrm{Tot}(R(X))$, leading to $R(X)\to\mathrm{Tot}(R(X))$ is Pr\"ufer, with $\mathrm{Tot}(R(X))$ zero-dimensional. Mimicking the proof we got for $R$, it follows that 
  $R(X)$ is treed and we can apply Lemma~\ref{TREED} 
  to get that $R$ is quasi-Pr\"uferian. 
\end{proof} 

\begin{remark} If we suppose that $\mathrm{dim}(\mathbb M(R))=0$, then $\mathbb M(R_P)\cong(\mathbb M(R))_P$. To see this it is enough to use \cite[Proposition 3.5, p.115]{L}. Suppose now in addition that $R\subset \mathbb M(R)$ is quasi-Pr\"ufer. Mimicking the proof of Proposition \ref{GENQPC}, we can also show that $R$ is quasi-Pr\"uferian. The reader may find in \cite{L} many contexts in which $\mathrm{Tot}(R)= \mathbb M( R)$. 
\end{remark}

\begin{example} \label{UFD}A total quotient ring (which is a (quasi-)Pr\"ufer ring) need not to be quasi-Pr\"uferian. To see this we consider the following example given by Lucas \cite[Example 2.11]{LU}. There is a total quotient ring $R$ which is not locally Pr\"ufer because there is a prime ideal $P$ of $R$ such that $R_P$ is not Pr\"ufer but is integrally closed (actually, an UFD). Suppose that quasi-Pr\"ufer rings are quasi-Pr\"uferian, then $R_P$ would be quasi-Pr\"uferian, since quasi-Pr\"ufer. As it is integrally closed, it would be Pr\"ufer, a contradiction.
\end{example}

\section{Pr\"ufer FCP extensions over a local ring}
   
Clearly, a minimal extension is a flat epimorphism if and only if it is Pr\"ufer. So we will call Pr\"ufer minimal such extensions. We note as a first result the following Proposition, which results from Proposition~\ref{AYACHE}. 

   \begin{proposition} A Pr\"ufer minimal extension with crucial ideal maximal $M$ is a QR-extension if and only if $M$ is equivalent  to a principal ideal. 
\end{proposition}
      
Proposition~\ref{VALU} take the following form, observing that a Pr\"ufer extension is integrally closed.
   
\begin{proposition}\label{VALBIS}\cite[Theorem 6.8 and Theorem 6.10]{DPP2} If $R$ is a local ring, an extension $R\subseteq S$ is Pr\"ufer FCP (resp.; minimal) if and only if there exists $P\in \mathrm{Spec}(R)$ such that $S=R_P$, $P=SP$ and $R/P$ is a finite dimensional (resp.; one-dimensional) valuation domain. Under these conditions, $S/P$ is the quotient field of $R/P$ and $P$ is a divided prime ideal of $R$. 
   \end{proposition}
   
The conductor of  a ring extension $R\subset S$ is denoted  by $(R:S)$. The following Corollary recalls, for a Pr\"ufer minimal extension $R\subset S$, the link between the crucial ideal maximal $\mathcal{C}(R,S)$ and  $(R:S)$.

\begin{corollary}\label{VAL4} If  $R\subset S$ is a Pr\"ufer minimal extension with crucial ideal maximal $M$,  then, $P: =(R:S)$ is a prime ideal of $R$,   $P \subset M$ and there is no prime ideal of $R$ contained strictly between $P$ and $M$.
    \end{corollary}
\begin{proof} First, $P:=(R:S)$ is a prime ideal of $R$ by \cite[Lemme 3.2]{FO}.   Moreover, $P_M=PR_M=(R_M:S_M)$ with $R_M\neq S_M$ shows that $P\subset M$. At last, $R_M\subset S_M$ is also a Pr\"ufer minimal extension. Then, according to Proposition \ref{VALBIS}, $R_M/P_M$ is a one-dimensional valuation domain, so that there is no prime ideal of $R_M$ contained strictly between $P_M$ and $MR_M$ giving that there is no prime ideal of $R$ contained strictly between $P$ and $M$.
    \end{proof}

\begin{corollary}\label{VALTER} Let $R\subset S$ be a Pr\"ufer minimal extension over a local ring $(R,M)$. Then, with the notation of Proposition~\ref{VALBIS}, each element $t\in M\setminus P$ is a strong divisor of $R$, $S \cong R_t$, $P= \cap [Rt^n | n \in \mathbb N]$  and $M= \sqrt{Rt}$.
    \end{corollary}
\begin{proof} Because $t \notin P$, we have $t\in \mathrm{U}(S)$ and then a factorization $R\subseteq R_t\subseteq S$, so that $S= R_t$ by minimality. By (Proposition~\ref{LOCALPRUF}), $t$ is a strong divisor. The third statement follows from Corollary~\ref{LPRUF}. Because $R/P$ is one-dimensional, $M$ is the only prime ideal containing $P$, so that $M=\sqrt{Rt}$.
    \end{proof}
  
\begin{proposition}\label{1.13} Let $R\subset S$ be a ring extension where $(R,M)$ is a local ring.  The following statements are equivalent:
 \begin{enumerate}
 
 \item $R\subset S$ is a Pr\"ufer minimal  extension;
 
\item there is a strong divisor $a\in R\setminus\mathrm{U}(R)$ such that $S=R_a$ and $\sqrt{Ra}\subseteq\sqrt{Rb}\Rightarrow\sqrt{Ra}=\sqrt{Rb}$ (or equivalently, $\mathrm{D}(a)\subseteq\mathrm{D}(b)\Rightarrow\mathrm{D}(a)= \mathrm{D}(b)$) for each $b\in R\setminus \mathrm{U}(R)$;

\item there is a strong divisor $a\in R\setminus\mathrm{U}(R)$ such that $S=R_a$  and $M = \sqrt{Ra}$;

\item there is a strong divisor $a\in R\setminus\mathrm{U}(R)$, such that $S=R_a$, and such that $\mathrm{D}(a)$ is an open affine subset, maximal in the set of  proper open affine subsets.
\end{enumerate}
 \end{proposition} 
\begin{proof} We have clearly (2) $\Leftrightarrow$ (3) by (Corollary~\ref{VALTER}), once (1) $\Leftrightarrow$ (2) is proved. 
 We then prove that (1) is equivalent to (2). Suppose that $R\subset S$ is a minimal extension, that is a flat epimorphism. Then $R\subset S$ is clearly a Pr\"ufer extension. By Proposition~\ref{1.8}, there is a mcs $\Sigma$ of $R$, whose elements are strong divisors and such that the extension identifies with $R\subset R_\Sigma$. Picking an arbitrary element $a \in \Sigma$, we get a factorization $R\to R_a \to R_\Sigma$. Its factors are injective because  the flat epimorphism $R\to R_a$ 
 verifies \cite[Lemme 3.4, p.114]{L}.
 
  As $a$ is not invertible, $R\neq R_a$ implies $S = R_a$, by minimality of $R\subset S$. In the same way a factorization $R\subset R_b \subseteq R_a$ implies $R_b = R_a$, or equivalently $\mathrm{D}(a) \subseteq \mathrm{D}(b) \Rightarrow \mathrm{D}(a)= \mathrm{D}(b)$, which means that $\sqrt {Ra} \subseteq \sqrt {Rb} \Rightarrow \sqrt {Ra} = \sqrt {Rb}$.
 
We prove the converse. Observe that for any mcs $\Sigma$ of $R$ and $a\in R$, such that there is a factorization $R \to R_\Sigma \to R_a$, we have $\mathrm{D}(a) \subseteq \cap [\mathrm{D}(\sigma) \mid \sigma \in \Sigma]$. Suppose that $R\subset S$ verifies the conditions of the proposition. Then $S=R_a$, where $a\in R\setminus \mathrm U(R)$ is a strong divisor and then $R\subset S$ is Pr\"ufer, so that any subextension $R \subset T \subseteq S$ is Pr\"ufer. By Proposition~\ref{1.8}, we get that $T=R_\Sigma$, for some mcs $\Sigma$ of $R$. The above observation shows that $\mathrm{D}(a) \subseteq \mathrm{D}(\sigma)$ for any $\sigma \in \Sigma$. The last condition entails that $\Sigma = <a>$, and then $T =R_a$. Therefore, $R\subset S$ is minimal and a flat epimorphism.
 
Clearly (4) implies (2). The converse is a consequence of the following facts. If $\mathrm{D}(a) \subseteq \mathrm{D}(I)$, where $\mathrm{D}(I)$ is an open affine subset different from $\mathrm{Spec}(R)$, we have a factorization $R \subset \Gamma (\mathrm{D}(I)) \subseteq R_a$ and $R \to \Gamma (\mathrm{D}(I))$ is an injective flat epimorphism whose spectral image is $\mathrm{D}(I)$.
\end{proof}
   
\begin{lemma}\label{1.22} Let $R$ be a ring and $a\in R$. Then, there exists some $M\in\mathrm{Max}(R)$ such that $M=\sqrt {Ra}$ if and only if $\mathrm{Supp}(R_a/R)=\{M\}$.
 \end{lemma}
\begin{proof} Let $M\in\mathrm{Max}(R)$. Then $M=\sqrt {Ra}\Leftrightarrow M$ is the only $P\in\mathrm{Spec}(R)$ containing $a\Leftrightarrow$ for any $P\in\mathrm{Spec}(R)\setminus\{M\},\ a\not\in P$ and $a\in M\Leftrightarrow$ for any $P\in\mathrm{Spec}(R)\setminus\{M\},\ a/1\in\mathrm{U}(R_P)$ and $a/1\not\in\mathrm{U}(R_M)\Leftrightarrow$ for any $P\in\mathrm{Spec}(R)\setminus\{M\},\ R_P=(R_a)_P$ and $R_M\neq(R_a)_M\Leftrightarrow \mathrm{Supp}(R_a/R)=\{M\}$. 
  \end{proof}
 
 \begin{proposition}\label{1.23} Let $R\subset S$ be a ring extension.  The following statements are equivalent:
 \begin{enumerate}
 
 \item $R\subset S$ is a  minimal  QR-extension;
 
\item there exists some $M\in\mathrm{Max}(R)$ such that $\mathrm{Supp}(S/R)=\{M\}$ and there is a lsd $a\in R\setminus\mathrm{U}(R)$ such that $S=R_a$;

\item  there is a lsd $a\in R\setminus\mathrm{U}(R)$ such that $S=R_a$ and there exists some $M\in\mathrm{Max}(R)$ such that $M=\sqrt {Ra}$;
\end{enumerate}
If these conditions are satisfied, then $M\not\subseteq\cup[P\in \mathrm{Max}(R)\mid P\neq M]$. 
 \end{proposition} 
\begin{proof} (1) $\Rightarrow$ (2) Assume that $R\subset S$ is a minimal  QR-extension. Then, there exists some $M\in\mathrm{Max}(R)$ such that $\mathrm{Supp}(S/R)=\{M\}$. Since $R\subset S$ is a minimal QR-extension, Theorem \ref{1.12.1} asserts that there is a lsd $a\in R\setminus\mathrm{U}(R)$ such that $S=R_a$, since $S$ is of finite type over $R$. It follows that $\mathrm{Supp}(R_a/R)=\{M\}$, which gives $M=\sqrt{Ra}$ by Lemma \ref{1.22}, so that $a\in M\setminus\cup[P\in\mathrm{Max}(R)\mid P\neq M]$, giving $M\not\subseteq\cup[P\in\mathrm{Max}(R)\mid P\neq M]$, proving the last assertion.

(2) $\Rightarrow$ (1) Assume that there exists some $M\in\mathrm{Max}(R)$ such that $\mathrm{Supp}(S/R)=\{M\}$ and there is a lsd $a\in R\setminus\mathrm{U}(R)$ such that $S=R_a$. These two conditions lead to $M=\sqrt{Ra}$ by Lemma \ref{1.22}. As $a/1$ is a strong divisor of $R_M$ and $S_M=(R_a)_M=(R_M)_{(a/1)}$, Proposition \ref{1.13} shows that $R_M\subset S_M$ is minimal Pr\"ufer. Moreover, $R_P=S_P$ for any $P\in\mathrm{Max}(R),\ P\neq M$ implies that $R\subset S$ is minimal Pr\"ufer. At last, Theorem \ref{1.12.1} shows that $R\subset S$ is a QR-extension, since minimal.

(2) $\Leftrightarrow$ (3) by Lemma \ref{1.22}.
 \end{proof}
 
 \begin{corollary} A minimal Pr\"ufer extension $R\subset S$ such that $S=R_s$ for some lsd $s$ of $R$ is a QR-extension.
 \end{corollary}
 
\begin{proof} Since $R\subset S$ is minimal, there exists some $M\in\mathrm{Max}(R)$ such that $\mathrm{Supp}(S/R)=\{M\}$. Then Proposition \ref{1.23} shows that  $R\subset S$ is a    QR-extension.
   \end{proof}
 
 \begin{example} Set $R:=\mathbb Z,\ P:=p\mathbb Z$ where $p$ is a prime integer and $S:=\mathbb Z_p$. Obviously, $\mathrm{Supp}(S/R)=\{P\}$, so that $S=R_p$ and $P=\sqrt{Rp}$. Then, $p/1$ is a strong divisor of $R_P$ and $p/1\in\mathrm{U}(R_M)$ for any $M\in\mathrm{Spec}(R)\setminus\{P\}$, 
  showing that $p$ is a lsd of $R$.
  We recover the fact that $\mathbb Z\subset\mathbb Z_p$ is a minimal Pr\"ufer QR-extension. 
  \end{example} 
   
Next result shows that Pr\"ufer FCP extensions can be described in a special manner.
   
   \begin{proposition}\label{1.15} 
  \cite[Proposition 1.3]{Pic 5}
    Let $R\subset S$ be an FCP extension. 
Then $R\subset S$ is integrally closed $\Leftrightarrow$ $R\subset S$ is Pr\"ufer $\Leftrightarrow$ $R\subset S$ is a tower of Pr\"ufer  minimal extensions.
\end{proposition}

\begin{theorem}\label{1.16} An FCP QR-extension $R\subseteq S$ admits a tower of Pr\"ufer minimal extensions $R\subset R_1\subset\cdots\subset R_i\subset R_{i+1}\subset\cdots\subset R_n=S$, where $R_{i+1}=(R_i)_{a_i}=R_{a_i}$ for some lsd $a_i\in R$ and $S=R_{a_1\ldots a_n}=R_{a_n}$. The integer $n$ is independent of the sequence and is equal to $|\mathrm{Supp}(S/R)|$.
   \end{theorem}
 
\begin{proof} There is a tower of Pr\"ufer minimal extensions $R_i\subset R_{i+1}$ by Proposition \ref{1.15} because a QR-extension is Pr\"ufer. Therefore, each $T\in [R,S]$ is a localization $R_a$, for some lsd $a\in R$ by Theorem \ref{1.12.1} and $R_i\subset R_{i+1}$ identifies to $R_i\subset R_{a_i}$ for some lsd $a_i\in R\setminus\mathrm{U}(R_i)$.
 
Then by minimality, we get that $R_{i+1}=R_{a_i}=(R_i)_{a_i}$ and the result follows.
 The last result is \cite[Proposition 6.12]{DPP2}.
 \end{proof}
The above result applies when $R\subseteq S$ is an FCP extension $A(X) \subseteq B(X)$, or equivalently, $A\subseteq B$ has FCP \cite[Theorem 3.9]{DPP3}.
 
\begin{proposition}\label{1.18} Let $R\subset S$ be an FCP Pr\"ufer extension over a local ring $(R,M)$. 
  
 \begin{enumerate}
\item There is a sequence of Pr\"ufer minimal extensions between local rings $R\subset R_1\subset\cdots\subset R_i\subset R_{i+1}\subset \cdots\subset R_n = S$, where $R_{i+1}=(R_i)_{a_i}=R_{a_i}$ for some $a_i\in \Delta(R)$ and $S=R_{a_1\ldots a_n}=R_{a_n}$. In fact $[R,S]=\{R_i\}_{i=0}^n$. 
 Moreover, $R\subset S$ is a QR-extension.

\item There is  some subset $\{P_0,P_1,\cdots,P_n\}$ of $\mathrm{Spec}(R)$, with $P_0=M$ and $P_i\subset P_{i-1}$ for $i=1,\cdots,n$, such that $R_i =R_{P_i}$, $P_iR_i=P_i$ and $R/P_i$ is a valuation domain whose dimension is $i$.
\item For all $i=1,\cdots,n$ and $t\in P_{i-1}\setminus P_i$, we have $R_i=R_t$. Moreover, each element of $M\setminus P_n$ is a strong divisor.
\item Any finitely generated $S$-regular ideal is equivalent to a principal ideal.
\end{enumerate}
 \end{proposition}
 \begin{proof} 
(1) We know that $R\subset S$ is chained because $R$ is local \cite[Theorem 6.10]{DPP2}, and each $R_i$ is local by Proposition~\ref{VALBIS}. 
Moreover, this proposition  shows that for each $i\in\mathbb N_n$, there exists $P_i\in\mathrm{Spec}(R)$ such that $R_i=R_{P_i},\ P_iR_i=P_i$ and $ R/P_i$ is a $i$-dimensional valuation domain. Therefore, $R\subset S$ is a QR-extension. Then apply Theorem~\ref{1.16}. 
    
(2) Since $R\subset R_1$ is Pr\"ufer minimal, $P_0:=M=\mathcal{C}(R,R_1)$. As $S=R_n$, it follows that $R/P_n$ is a $n$-dimensional valuation domain and $\{P_0,P_1,\cdots,P_{n-1}\}:=\mathrm{Supp}(S/R)$ according to \cite[Proposition 6.12]{DPP2} with $P_i\subset P_{i-1}$ for each $i\in\mathbb N_n$.
  
(3) Let $t\in P_{i-1}\setminus P_i$. Then, $t$ is a unit of $R_{P_i}=R_i$ and $ R_t\subseteq R_i$. As $t\in P_{i-1}$, this implies that $t$ is a not a unit of $ R_{P_{i-1}}=R_{i-1}$, so that $R_t\not\subseteq R_{i-1}$. But $[R,S]$ is a chain, which leads to $R_t=R_i$. 
  
Let $x\in M\setminus P_n$. Since $\{P_0,P_1,\cdots,P_n\}$ is a chain, there exists some $i\in\mathbb N_n$ such that $x\in P_{i-1}\setminus P_i$. Then, $ R_x=R_i$ by the first part of (3). To end, Proposition \ref{LOCALPRUF} and Definition \ref{1.11} show that $x$ is a strong divisor.

(4) Since $R\subset S$ is a QR-extension, any finitely generated $S$-regular ideal is equivalent to a principal ideal according to the recall at the beginning of Section 8.
\end{proof}

 We end this section by a generalization of Proposition \ref{1.18}
to  $\mathcal B$-extensions. We recall that an extension $R\subset S$ is a {\em $\mathcal B$-extension} if the map $\beta : [R,S]\to\prod [[R_M,S_M]\mid M\in \mathrm{MSupp}(S/R)]$ defined by $T\mapsto(T_M)_{M\in\mathrm{MSupp}(S/R)}$ is bijective. Actually, an FCP extension $R\subseteq S$ is a $\mathcal B$-extension if and only if $R/P$ is a local ring for each $P\in\mathrm{Supp}(S/R)$ \cite[ Proposition 2.21]{Pic 10}. The following lemma gives a first special case of $\mathcal B$-extension.
For any extension $R\subseteq S$, the {\it length} $\ell[R,S]$ of $[R,S]$ is the supremum of the lengths of chains of $R$-subalgebras of $S$. Notice that if $R\subseteq S$ has FCP, then there {\it does} exist some maximal chain of $R$-subalgebras of $S$ with length $\ell[R,S]$ \cite[Theorem 4.11]{DPP3}.

\begin{lemma}\label{1.19} Let $R\subset S$ be an FCP Pr\"ufer extension such that $|\mathrm{MSupp}(S/R)|=1$. Then, $R\subset S$ is a $\mathcal B$-extension where $\mathrm{Supp}(S/R)$ and $[R,S]$ are chains with $n:= |\mathrm{Supp}(S/R)|=|[R,S]|-1$. There is a tower of Pr\"ufer minimal extensions $R\subset R_1\subset\cdots\subset R_i\subset R_{i+1}\subset\cdots\subset R_n=S$,  such that $[R,S]=\{R_i\}_{i=0}^n$. We define as follows a subset $\{P_0,P_1,\ldots,P_n\}$ of $\mathrm{Spec}(R)$ by $\mathrm{Supp}(S/R)=\{P_0,P_1,\cdots,P_{n-1}\}$ and $P_n:=R\cap(R_{n-1}:S)$, where $P_i\subset P_{i-1}$ for each $i\in\mathbb N_n$. In particular, $\mathrm{Supp}(R_i/R)=\{P_0,P_1,\cdots,P_{i-1}\}$. 
  \end{lemma}
\begin{proof} From $|\mathrm{MSupp}(S/R)|=1$, we deduce that $R\subset S$ is a $\mathcal B$-extension by \cite[Proposition 2.21]{Pic 10}. If $\{M\}:=\mathrm{MSupp}(S/R)$, the map $\beta:[R,S]\to[R_M,S_M]$ defined by $\beta(T)=T_M$ for any $T\in[R,S]$ is bijective. But, $R_M\subset S_M$ is chained by Proposition \ref{1.18}, whence $R\subset S$ is chained. According to \cite[Theorem 3.10]{Pic 15}, $\mathrm{Supp}(S/R)$ has a least element $P$ and $\mathrm{Supp}(S/R)=\mathrm{V}(P)$ is chained. Moreover, $|\mathrm{Supp}(S/R)|=|[R,S]|-1=\ell[R,S]$ by \cite[Proposition 6.12]{DPP2}. If  $n:=|[R,S]|-1$, there is a sequence of Pr\"ufer minimal extensions $R\subset R_1\subset\cdots\subset R_i\subset R_{i+1}\subset\cdots\subset R_n=S$ such that $[R,S]=\{R_i\}_{i=0}^n$ since $R\subset S$ is chained. Moreover, there is some subset $\{P_0,P_1,\cdots,P_n\}$ of $\mathrm{Spec}(R)$ such that $\mathrm{Supp}(S/R)=\{P_0,P_1,\cdots,P_{n-1}\}$ where $P_i\subset P_{i-1}$ for each $i\in\mathbb N_{n-1}$. In particular, $\mathrm{Supp}(R_i/R)=\{P_0,P_1,\cdots,P_{i-1}\}$ for each $i\in\mathbb N_n$. In fact, we have $P_i=R\cap\mathcal{C}(R_i,R_{i+1})$ for each $i=0,\cdots,n-1$ by \cite[Corollary 3.2]{DPP2} and also $P_i=R\cap(R_{i-1}:R_i)$ for each $i\in\mathbb N_{n-1}$ by Corollary \ref{VAL4}. 
\end{proof}

 \begin{proposition}\label{1.21} Let $R\subset S$ be an FCP Pr\"ufer extension such that $|\mathrm{MSupp}(S/R)|=1$. Set $[R,S]=
\{R_i\}_{i=0}^n,\ P_n:=R\cap(R_{n-1}:S)$ and $\mathrm{Supp}(S/R)=\{P_0, P_1,\cdots,P_{n-1}\}$ as defined in Lemma \ref {1.19}. The following conditions are equivalent:
\begin{enumerate}
\item $R\subset S$ is a QR-extension;
\item for each $i\in\mathbb N_n$, there is some lsd $a_i\in R$ such that  $R_i=R_{a_i}$;
\item for each $i\in\mathbb N_n$, there is some $a_i\in R$ such that $P_{i-1}=\sqrt{Ra_i}$.
\end{enumerate}
If these conditions hold, then, for each $i\in\mathbb N_n$, we have $R_i= (R_{i-1})_{a_i}$ and $a_i\in P_{i-1}\setminus P_i$. Moreover, $a_i$ satisfies (2) if and only if it satisfies (3).
\end{proposition}

\begin{proof} (1) $\Leftrightarrow$ (2) by Theorem \ref{1.12.1} and (3) $\Leftrightarrow$ (1) by Proposition \ref{AYACHE}.

Assume that (2) holds with $a_i$ an lsd such that $R_i=R_{a_i}$. Obviously, $R_i=(R_{i-1})_{a_i}\ (*)$.
 
Let $P\in\mathrm{Spec}(R)$. Then, $a_i\in P$ implies that $PR_i=R_i$, so that $P\in\mathrm{Supp}(R_i/R)=\{P_0,P_1,\cdots,P_{i-1}\}$. Then, there is some $j<i$ such that $P=P_j$ and $a_i\not\in P_i$. To prove that $a_i\in P_{i-1}$, we localize the extension $R\subset S$ at $M$. Set $R'_j:=(R_j)_M,\ P_j':= P_jR_M$ for each $j=0,\cdots,n$ and $a'_i:=a_i/1$. Then, $R'\subset S'$ is an FCP Pr\"ufer extension over the local ring $(R',M')$ with $n=\ell[R',S']$. Using Proposition \ref{1.18}, we get that $\{P'_0,P'_1,\cdots,P'_n\}$ is the subset of $\mathrm{Spec}(R')$ such that $R'_i=R'_{P'_i}$, $P'_iR'_i=P'_i$. It follows that $(*)$ gives $R'_i=R'_{a'_i}=(R'_{i-1})_{a'_i}$, with $R'_{i-1}$ a local ring with maximal ideal $P'_{i-1}$. Then $a'_i\in P'_{i-1}$ since $R'_i= (R'_{i-1})_{a'_i}\neq R'_{i-1}$ which gives $a_i\in P_{i-1}\setminus P_i$. To conclude, $P_{i-1}=\sqrt{Ra_i}$ as the least prime ideal of $R$ containing $a_i$, and also the least element of $\mathrm{Supp}(R_i/R)$. It follows that $a_i$ satisfies (3).
 
Conversely, if there is some $b_i$ such that $P_{i-1}=\sqrt{Rb_i}$, then $\sqrt{Ra_i}=\sqrt{Rb_i}$ implies $R_i=R_{a_i}=R_{b_i}$ by Definition \ref{1.11}.
\end{proof}

In order to generalize Lemma \ref{1.19} to an arbitrary FCP Pr\"ufer $\mathcal B$-extension, we need the following definition introduced in \cite{Pic 15}. Let $R\subseteq S$ be an FCP $\mathcal B$-extension and $M\in\mathrm {MSupp}(S/R)$. The elementary splitter $\sigma(M):=T$, associated to $M$, is defined by $\mathrm{MSupp}(T/R)=\{M\}$ and $\mathrm{MSupp}(S/T)=\mathrm{MSupp}(S/R)\setminus\{M\}$. Such a $T$ always exists (see \cite[Theorem 4.6 and the paragraph after Corollary 5.5]{Pic 15}).

\begin{proposition}\label{1.20} Let $R\subset S$ be an FCP Pr\"ufer $\mathcal B$-extension and QR-extension. Let $T\in]R,S]$ and set $\mathrm{MSupp}(T/R)=:\{M_1,\ldots,M_n\}$. For each $i\in\mathbb N_n$, let $P_i$ be the least element of $\mathrm{Supp}(T/R)$ contained in $M_i$. Then, there exists some lsd $t\in R$ such that $T=R_t$ and $\sqrt{Rt}=\cap[P_i\mid i\in\mathbb N_n]$.
   \end{proposition}
\begin{proof}  Since $R\subset S$ is an FCP Pr\"ufer $\mathcal B$- and QR-extension, so is $R\subset T$ by \cite[Proposition 3.5]{Pic 15} for the $\mathcal B$-extension property. Set $\mathrm{MSupp}(T/R)=:\{M_1,\ldots,M_n\}$. For each $i\in\mathbb N_n$, according to \cite[Theorem 3.10]{Pic 15}, there is a least element $P_i$ of $\mathrm{Supp}(T/R)$ contained in $M_i$. The same reference gives that $\mathrm{V}(P_i)$ is a chain, whose greatest element is $M_i$ and  least element is $P_i$ .

For each $i\in\mathbb N_n$, set $T_i:=\sigma(M_i)$, so that $|\mathrm {MSupp}(T_i/R)|=1$. Moreover, for any $P\in\mathrm{V}(P_i)$, we have $P\in\mathrm{Supp}(T_i/R)$.
 Then, $\mathrm{Supp}(T_i/R)=\mathrm{V}(P_i)$ because $\mathrm{MSupp}(T_i/R)=\{M_i\}$, and we can apply Proposition  \ref{1.21}.
   For each $i\in\mathbb N_n$, we have $T_i=R_{a_i}$ for some lsd $a_i\in R$ and $P_i=\sqrt{Ra_i}$ for each $i\in\mathbb N_n$. Set $t:=a_1\cdots a_n$, which is still an lsd. Now, \cite[Proposition 5.11]{Pic 15} asserts that $T=\prod [T_i\mid i\in\mathbb N_n]=\prod[R_{a_i}\mid i\in\mathbb N_n]=R_t$. Moreover, $\sqrt{Rt}=\sqrt{\prod[R_{a_i}\mid i\in\mathbb N_n]}=\cap[\sqrt{R_{a_i}}\mid i\in\mathbb N_n]=\cap[P_i\mid i\in\mathbb N_n]$.
\end{proof}
 
 \section{The set of all primitive elements}
Let $R\subseteq S$ be an extension. An element $s\in S$ is called {\it primitive} (over $R$) if there exists a polynomial $p(X)\in R[X]$ whose content is $R$ and such that $p(s)=0$.  
An extension $R\subseteq S$ is called a $\mathcal P$-extension if all the elements of $S$ are primitive over $R$. Important examples of $\mathcal P$-extensions are given by the Pr\"ufer extensions of \cite{KZ} (equivalently normal pairs \cite[Theorem 5.2, p.47]{KZ}). We recall that an element $s$ of $S$ is primitive if and only if $R\subseteq R[s]$ has the INC property \cite[Theorem 2.3]{DLO} and that an extension $R\subseteq S$ is a $\mathcal P$-extension if and only if $R\subseteq S$ is a INC pair \cite[Corollary  2.4]{DLO}. We proved that an INC pair is nothing but a quasi-Pr\"ufer extension \cite[Theorem 2.3]{Pic 5}.  

It follows easily that an extension is quasi-Pr\"ufer if and only if it is a $\mathcal P$-extension. Therefore an FCP extension is a $\mathcal P$-extension \cite[Corollary 3.4]{Pic 5}. 

For an extension $R\subseteq S$, we denote by $\mathcal P_S(R)$ the set of all elements of $S$ that are primitive over $R$,  
 a set studied by Dobbs-Houston in \cite{DH}.

 We defined in \cite[Theorem 3.18]{Pic 5} the quasi-Pr\"ufer closure (or hull) $\overset{\Longrightarrow}R$ of an extension $R\subseteq S$. This closure is the greatest quasi-Pr\"ufer subextension $R\subseteq T$ of $R\subseteq S$ and is equal to $\widetilde{\overline R}$. It follows that $\overset{\Longrightarrow} R\subseteq\mathcal P_S(R)$. Obviously, $\overline R\subseteq \mathcal P_S(R)$.

 \begin{proposition} Let $R\subseteq S$ be an extension. Then $\mathcal P_S(R)$ is a ring if and only if $\mathcal P_S(R)=\overset{\Longrightarrow} R $.
 \end{proposition}
\begin{proof} It is enough to show that if $\mathcal P_S(R)$  is  a ring, it is contained in $\overset{\Longrightarrow}R $, that is $\overline R\subseteq \mathcal P_S(R)$ is Pr\"ufer. But we may assume that $R\subset S$ is integrally closed 
 because an element of $S$ primitive over $R$ is primitive over $\overline R$.

 By \cite[Proposition 2.6]{DH}, we can  also assume that $R$ is local. Let $s\in\mathcal P_S(R)$. Then, either $s\in R$ or $s$ is a unit of $S$ such that $s^{-1}\in R$ 
  according to \cite[Corollary 2.5]{DH}.
  If $s\in R$, then $R[s]=R[s^2]$. If $s$ is a unit of $S$ such that $s^{-1}\in R$, then $s^{-1}\in R[s^2]$ implies $s\in R[s^2]$, and we still have $R[s]=R[s^2]$.
Therefore, for each $s\in\mathcal P_S(R)$, we have $R[s^2]=R[s]$. We deduce from \cite[Chapter 1, Theorem 5.2]{KZ}, that the ring $\mathcal P_S(R)$ is Pr\"ufer over $R$ and then $\mathcal P_S(R)=\overset{\Longrightarrow} R $.
\end{proof}

It may happen that an extension $R\subseteq S$ is such that $\mathcal P_S(R)=S$. For example, if we denote by $\mathrm{Tot}(R)$ the total ring of quotients of a ring $R$, then $\mathcal P_{\mathrm{Tot}(R)}R$ is a ring if and only if $\mathcal P_{\mathrm{Tot}(R)}R=\mathrm{Tot}(R)$  \cite[Corollary 2.9]{DH}. If $R\subseteq S$ defines a lying-over pair, then $\mathcal P_S(R)=\overset{\Longrightarrow}R$ \cite[Proposition 3.11]{DH}. The next result generalizes some results of \cite{DH}. 

\begin{corollary} \label{PRI}An extension $R\subseteq S$ is  such that $\mathcal P_S(R) = S$ if and only if $R\subseteq S$ is quasi-Pr\"ufer.
\end{corollary}

\end{document}